\documentclass[10pt]{article}
\usepackage{amsmath}
\usepackage{mathtools}
\usepackage{amssymb}
\usepackage{xcolor}
\usepackage{graphicx} 
\usepackage{algorithm}
\usepackage{algpseudocode}
\usepackage{amsmath}
\usepackage{amsfonts}
\usepackage{amsthm}
\usepackage{graphicx}
\usepackage{caption,subcaption}
\usepackage{url}
\usepackage{xfrac}
\usepackage{diagbox}
\usepackage{pgfplots}
\usetikzlibrary{calc}
\usepackage{hyperref}
\usepackage{multirow}

 \newtheorem{theorem}{Theorem}
	\newtheorem{example}[theorem]{Example}
	
	\newtheorem{lemma}[theorem]{Lemma}
	\newtheorem{remark}[theorem]{Remark}
	\newtheorem{definition}[theorem]{Definition}

\newcommand{\norm}[1]{\left\lVert #1 \right\rVert}

\newcommand{\rev}[1]{\textcolor{black}{#1}}

\title{Computing the density of the Kesten–Stigum limit\\ in supercritical Galton–Watson processes}

\author{Alice Cortinovis\thanks{Department of Computer Science, University of Pisa, PI, Italy and member of INdAM / GNCS. Email: {\tt  alice.cortinovis@unipi.it}} \and Sophie Hautphenne\thanks{School of Mathematics and Statistics, University of Melbourne, Australia. Email: {\tt sophiemh@unimelb.edu.au}} \and Stefano Massei\thanks{Department of Mathematics, University of Pisa, PI, Italy and member of INdAM / GNCS. Email: {\tt  stefano.massei@unipi.it}}}

\date{}
\begin{document}

\maketitle

\begin{abstract}

This paper proposes a novel numerical method for computing the density of the Kesten-Stigum limit random variable associated with a supercritical Galton-Watson process. This random variable captures the cumulative effect of early demographic fluctuations and determines the random amplitude governing long-term exponential population growth. 
Beyond classical branching process theory, the Kesten-Stigum limit plays a central role in scaling limits of density-dependent population models, where it induces random initial conditions for deterministic fluid approximations, and in statistical inference problems for growing populations. Despite its importance, computing its density in a stable and efficient manner for general offspring distributions remains a significant challenge. 
Our approach leverages the functional equation satisfied by the Laplace-Stieltjes transform of the limit distribution and combines it with a moment-matching reconstruction based on Laguerre polynomials with exponential damping. The resulting method is computationally efficient and applies to offspring distributions with bounded support. Its effectiveness is demonstrated on several numerical examples, including biologically motivated case studies.

\end{abstract}
\section{Introduction}
Galton-Watson (GW) branching processes are simple stochastic models used to describe the evolution of a population over discrete generations. In these models, each individual in a given generation produces a random number of offspring independently of all other individuals; see, e.g.,~\cite{Athreya2012,Harris1963}. More formally, a GW process is a  discrete-time Markov chain $\{Z_n\}_{n\in\mathbb N}$ taking values in the set of nonnegative integers, where $Z_n$ represents the population size at generation $n$. If $Z_n=0$ for some $n$, then $Z_k=0$ for all $k>n$, so that $0$ is an absorbing state. Otherwise, each individual in generation $n$ lives for one unit of time and, at the end of its life, gives birth to a random number $\theta$ of offspring. The distribution of $\theta$, called the \emph{offspring distribution}, is assumed to be independent of the generation and of the current population size. 

A GW process is fully characterized by the probability generating function of the offspring distribution (also called \emph{offspring generating function}), defined as
$$
P(z):= \sum_{j=0}^\infty p_j z^j, \qquad p_j := \mathbb P(\theta = j),
$$
where $p_j\ge 0$ for all $j\in\mathbb N$ and $\sum_{j=0}^\infty p_j=1$.

Given an initial population size $Z_0$, the evolution of the process is described by the recursion
$$
Z_n=\sum_{i=1}^{Z_{n-1}} \theta_{i}^{(n)}, \qquad n \ge 1,
$$
where $\{\theta_{i}^{(n)}\}_{i,n}$ are independent and identically distributed random variables with the same distribution as $\theta$. 

The long-term behavior of the population is governed by the mean offspring number
$$
m := \mathbb E[\theta]=\sum_{j=1}^\infty j p_j= P'(1),
$$
and depending on whether $m$ exceeds $1$ or not, different quantities are of interest, and methods for their computation have been proposed. 
When $m\le 1$, the population becomes extinct almost surely ($\lim_{n\to\infty} Z_n = 0$ with probability $1$) and the goal is  computing  the quasi-stationary distribution~\cite{hautphenne2020low}, i.e., the stationary distribution conditioned on the event of non extinction. 

The case of interest here is when $m > 1$, and the GW process is said \emph{supercritical}. The population of a supercritical GW process has a probability $q \in [0,1)$ of eventual extinction and a strictly positive probability $1-q$ of growing without bound, in the sense that $\lim_{n\to\infty} Z_n = \infty$ on the event of non-extinction. In this context, one is usually interested in computing the limit distribution of the process normalized by its mean; see the next sections for more details and applications. 

\subsection{Supercritical GW branching processes}\label{sec:supercritical}
In this paper, we focus on supercritical GW branching processes and assume that the process starts with a single individual, i.e., $Z_0=1$. Under this assumption, the extinction probability $q$ is the smallest nonnegative real solution of the equation $z=P(z)$; see~\cite[Chapter 1, Theorem 6.1]{Harris1963}. In particular, if the offspring distribution satisfies $P(0)=p_0=0$, then $q=0$. 

The expected population size at generation $n$ is given by
$$
\mathbb E[Z_n]=m^n,
$$
so that the mean population grows at a geometric rate, in accordance with the Malthusian law of growth. 

In many cases of interest, the rescaled process $\{Z_n/m^n\}_{n\in\mathbb N}$ converges almost surely to a nonnegative random variable $W$. This classical result is known as the \textit{Kesten--Stigum theorem}.  The existence of a non-trivial limit is ensured under mild moment conditions on the offspring distribution; in particular, finite variance, and in fact the weaker condition $\mathbb E[\theta\log\theta]<\infty$, are sufficient (see \cite[Theorem 6.2, Theorem 10.1]{Athreya2012}). In this setting, the limit random variable satisfies $\mathbb E[W]=1$ and $\mathbb P(W=0)=q<1$. For large values of $n$, the population size admits the approximation 
\begin{equation}\label{eq:WZn}
Z_n \approx W m^n,
\end{equation}
so that $W$ captures the cumulative effect of early stochastic fluctuations on the long-term growth of the population. 

\begin{remark}\label{rem:Z0}
If the process starts with $Z_0 = k > 1$ individuals, then the population can be viewed as the superposition of $k$ independent GW processes, each started from a single ancestor. In this case, the rescaled population size satisfies
\begin{equation}\label{Wk}
\frac{Z_n}{m^n} \;\xrightarrow{\text{a.s.}}\; \sum_{i=1}^k W_i=:W(k),
\end{equation}
where $\{W_i\}_{i=1}^k$ are independent copies of $W$. As a result, the distribution of $W$ captures the random component of long-term population growth arising from early stochastic fluctuations, for any finite initial population size.
\end{remark}

The distribution of $W$, and in particular its density $f(x)$, can be used to derive biologically relevant quantities for a newly founded ecological population, conditional on non-extinction. Two such applications are described below.

\paragraph{Distribution of the establishment time.}
Let $K>0$ denote a pre-defined establishment threshold, for example the minimum population size required for long-term viability.  Motivated by the approximation \eqref{eq:WZn} for large $n$, we introduce the continuous random variable
$$
\tau_K:= \frac{\log K - \log W}{\log m},
$$
so that the discrete establishment (hitting) time $T_K:=\min\{n\in\mathbb N: Z_n\ge K\}$ is heuristically approximated by $T_K\approx \lceil \tau_K\rceil$ (or by $\max\{0,\lceil \tau_K\rceil\}$ if needed).



Since $\log m > 0$, $\tau_K$ is a decreasing function of $W$: colonies with larger values of $W$ establish earlier.  
Working conditional on non-extinction (that is, $W>0$) with conditional density of $W \mid W > 0$ given by $f_+(w) = f(w)/(1-q)$ and corresponding cumulative distribution function $F_+(w)$, we obtain, for any $t \ge 0$,
\begin{align*}
\mathbb{P}(\tau_K \le t \mid W>0)
  &= \mathbb{P}\!\left(\frac{\log K - \log W}{\log m} \le t \,\middle|\, W>0\right) \\
  &= \mathbb{P}(W \ge K m^{-t} \mid W>0)
   = 1 - F_+\!\big(K m^{-t}\big).
\end{align*}
Differentiating with respect to $t$ gives the corresponding approximate density for the time at which the population first hits size $K$: 
\begin{equation}\label{eq:density_establishment_time}
g_{\tau_K}(t)
  \approx f_+\!\big(K m^{-t}\big)\, K m^{-t}\, \ln m.
\end{equation}
Hence, once $f$ (and thus $f_+$ and $F_+$) have been computed numerically, the distribution of the establishment time $T_K$ can be approximated directly. The latter is an asymptotic approximation, valid in the exponential-growth regime where $Z_n/m^n \approx W$. It becomes exact in the limit as $K\to\infty$ (corresponding to establishment occurring far in the exponential-growth regime), while for small thresholds, it provides qualitative insights into the variability of establishment times. 

\paragraph{Distribution of the population size at a fixed time.}
For large generation index $n$, the asymptotic approximation~\eqref{eq:WZn} implies
$$
\mathbb{P}(Z_n \le x \mid W>0) \approx F_+\!\left(\frac{x}{m^n}\right).
$$
In particular, if $Q_p$ denotes the $p$th quantile of $W \mid W>0$, then the corresponding quantile of $Z_n$ is approximately $m^n Q_p$.  
This provides a straightforward way to obtain approximate prediction intervals for future colony sizes: for instance, for large $n$, a 90\% prediction interval for $Z_n$ is given by
$$
[m^n Q_{0.05},\, m^n Q_{0.95}].
$$
As another consequence of the same approximation, for any threshold $K>0$, the probability that the population exceeds $K$ individuals by generation $n$ is
$$
\mathbb{P}(Z_n \ge K \mid W>0)
  \approx 1 - F_+\!\left(\frac{K}{m^n}\right).
$$


The importance of the limit random variable $W$ extends well beyond classical branching process theory.
In particular, $W$ appears naturally in scaling limits of density-dependent population models with large carrying capacities. When the population starts from a small number of individuals relative to the carrying capacity, early stochastic growth phases give rise to a random initial condition for the subsequent deterministic dynamics, obtained through a suitable transformation of $W$; see, e.g.,
\cite{barbour2016emergence,baker2018persistence,bauman2023approximation}.
A related phenomenon arises in branching models motivated by real-time PCR data, where the initial population size $Z_0=k$ is inferred from observations of a density-dependent process
\cite{chigansky2018can}. In that work, the initial population is assumed to be large, so that $W(k)$ in \eqref{Wk} is approximately normal. When this assumption is relaxed and the initial population size is small, the full distribution of $W$ -- and of its convolutions when $k>1$ -- becomes essential for estimating the initial population size. This further motivates the numerical approximation of the density of $W$ developed in this paper.

\subsection{Main contributions}

The main goal of this work is to provide a computational method for approximating the density function $f(x)$ of the random variable $W$. The analytical computation of $f(x)$ is only possible for a few simple models, so numerical methods are needed to address more general scenarios involving wider classes of offspring distributions. 

Our approach leverages the fact that the Laplace-Stieltjes transform of $W$, $\varphi(z):=\mathbb E[e^{-zW}]$, verifies the Poincaré functional equation
\begin{equation}\label{eq:poincare}
\begin{cases}
\varphi (mz)= P(\varphi(z)),\\
\varphi(0)=1,\quad \varphi'(0)=-1
\end{cases}
\end{equation}
for all $z$ with $\mathrm{Re}(z) \ge 0$ (where we note that $\varphi'(0)=-\mathbb E[W]$); see, e.g., Theorem 3 in~\cite[Chapter 1, Part C, Section 10]{Athreya2012}.
We focus on the case in which $P(z)$ is a polynomial of arbitrary degree $d\in\mathbb N$, corresponding to an offspring distribution with finite support. In this case, the Poincar\'e equation~\eqref{eq:poincare} holds true on the whole complex plane, has a unique solution $\varphi(z)$, and the Laplace-Stieltjes transform $\varphi(z)$ is an entire function. These properties follow from classical work of Poincaré and Valiron~\cite{Poincare1890,Valiron1954}; see also \cite{Derfel2007} for a more recent work on the topic\footnote{The result stated in~\cite[Chapter 7]{Valiron1954} applies to the slightly different boundary condition $\varphi(0) = 0$. We can recover our case~\eqref{eq:poincare} with a change of variables.}.

We first show that it is possible to recover, in theory, any number of coefficients of $\varphi(z)$ using a forward substitution method. Since this strategy sometimes results in low numerical accuracy, we also propose to approximate a truncated Taylor expansion of the Laplace-Stieltjes transform $\varphi(z)$ by means of  fixed-point functional iterations. This provides a number of moments of $W$, which we then use to fit an approximant to the density function $f(x)$. More specifically, we use a linear combination of generalized Laguerre polynomials, multiplied by an exponential damping factor, and we determine the coefficients by minimizing a moment-matching loss function. 

\subsection{Related work}

To the best of our knowledge, the numerical approximation of the distribution of the Kesten--Stigum limit $W$ for supercritical branching process has been addressed only recently in~\cite{Morris2024}. In that work, the authors consider continuous-time, multitype branching processes arising as early-stage approximations of stochastic
population models. 
Their methodology consists of two main steps:
\begin{enumerate}
\item Computing moments of $W$ via the derivatives of an analogue of the Poincar\'e equation~\eqref{eq:poincare} for continuous-time GW processes;
\item Approximating the conditional distribution of $W$ given non-extinction by fitting
a generalized Gamma distribution via a moment-matching criterion.
\end{enumerate}
While effective for the classes of models considered in~\cite{Morris2024}, this approach
comes with some limitations that we aim to address here. First, the explicit moment
recursions rely on the assumption that the progeny generating function $P(z)$ is of at most
quadratic form, whereas our method applies to polynomial $P(z)$ of arbitrary degree. Second, restricting the continuous component of the distribution of $W$
to a generalized Gamma family can be too rigid to capture certain qualitative features
of $W$, as illustrated in Section~\ref{sec:gamma}. 
On the other hand, the approach of~\cite{Morris2024} applies to multi-type GW processes, while we restrict our study to single-type processes.

\subsection{Outline}

The outline of the paper is the following. In Section~\ref{sec:poincare} we discuss how to compute or approximate the first $N+1$ terms of the Taylor expansion of $\varphi(z)$; in turn, these terms provide approximations of the first $N$ moments of $W$. In Section~\ref{sec:momentmatching} we use a moment-matching technique to obtain a linear combination of generalized Laguerre polynomials that approximates $f(x)$. Section~\ref{sec:numericalexamples} illustrates the performance of our numerical method on a variety of examples.

\section{Solving the Poincaré equation}\label{sec:poincare}
The first step towards estimating  $f(x)$ is computing an approximation of the Laplace-Stieltjes transform $\varphi(z) = \mathbb{E}[e^{-zW}]$ of the random variable $W$; our idea is to exploit the fact that $\varphi(z)$ is an entire function and solves the nonlinear functional equation with boundary conditions~\eqref{eq:poincare}.  In particular, for any $z \in \mathbb{C}$, we can express $\varphi(z)$ via its Taylor expansion around $z = 0$: 
$$\varphi(z)=\sum_{j=0}^{\infty}\varphi_j z^j.$$
In our setting, $\varphi_0=\varphi(0)=1$ and $\varphi_1=\varphi'(0)=-\mathbb E[W]=-1$, so that
$$
\varphi(z)=1 - z + \varphi_2 z^2 + \varphi_3 z^3 + \dots .
$$
Substituting this expansion into~\eqref{eq:poincare} yields an infinite system of nonlinear equations for the coefficients $\varphi_j$, $j\ge 2$.  
Moreover, there is a direct link between the coefficients $\varphi_j$ and the moments of $W$, see~\eqref{eq:phimoments} below. 

To write down the system explicitly, we leverage the fact that the ring of power series is isomorphic to that of semi-infinite lower triangular Toeplitz matrices, equipped with the usual matrix addition and multiplication. Specifically, we can identify $\varphi(z)$ with 
the semi-infinite Toeplitz matrix  $T_{\varphi}$ defined as
$$
T_{\varphi}:=\begin{bmatrix}
    \varphi_0\\
    \varphi_1&\varphi_0\\
    \varphi_2&\varphi_1&\varphi_0\\
    \vdots&\ddots&\ddots&\ddots
\end{bmatrix},
$$
and the powers $\varphi(z)^j$ with the matrix powers $T_{\varphi}^j$, for any $j\in\mathbb N$. Let us consider the case where $$P(z)=p_0+p_1z+\ldots+p_dz^d$$ is a polynomial of degree $d$, and let $e_1$ be the semi-infinite vector whose first entry is $1$ and whose remaining
entries are zero. Then the coefficients of the series expansion of $P(\varphi(z))$ can be
written as
\begin{equation}\label{eq:toeplitz-krylov}
    \begin{bmatrix} e_1 & T_{\varphi}e_1 & T_{\varphi}^2e_1 & \cdots & T_{\varphi}^de_1 \end{bmatrix} \begin{bmatrix} p_0 \\ p_1 \\ \vdots \\ p_d \end{bmatrix} =: A^{(\varphi)} p.
\end{equation}
Note that the matrix $A^{(\varphi)}$ has infinitely many rows and $d+1$ columns, and it is a  Krylov matrix generated by $T_\varphi$ and $e_1$; the vector $p$ contains the coefficients of $P(z)$. 

Finally, with a slight abuse of notation, let $\varphi$ denote the vector of Taylor
coefficients $(\varphi_j)_{j\ge0}$, and let $D$ be the semi-infinite diagonal matrix with
entries $D_{j,j}=m^j$ for $j=0,1,2,\ldots$. Then $D\varphi$ contains the coefficients
of the series expansion of $\varphi(mz)$, and the Poincar\'e equation can be rewritten in the equivalent matrix form
\begin{equation}\label{eq:infinite-fixed-point}
    D\varphi = A^{(\varphi)} p .
\end{equation}
Despite being infinite and non linear, the system~\eqref{eq:infinite-fixed-point} can be solved explicitly via a forward substitution strategy that we detail in the next section. However, the latter approach might be less accurate than iterative strategies when the degree of $P$ is moderate to high, so we also analyze two methods that approximate $\varphi(z)$ via functional fixed-point iterations expressed in terms of Taylor coefficients.  
The implementation of such iterative schemes requires truncating the  expansion of intermediate results to a finite order, and returns a finite number of Taylor coefficients $\varphi_0,\dots, \varphi_N$ of $\varphi(z)$. These methods are discrete counterparts of two functional fixed-point iterations. The first one is derived from a functional iteration that we prove to be globally convergent in the sense of Hardy space norm. The second one is obtained by applying Newton's method to \eqref{eq:poincare}, and typically exhibits a faster rate of convergence.

\subsection{A forward substitution method}\label{sec:forward}

We first observe that the nonlinear system \eqref{eq:infinite-fixed-point} has the following property: the $s$th equation depends only on $\varphi_0,\dots,\varphi_{s-1}$ and is linear in $\varphi_{s-1}$ for all $s\ge 3$. Together with the boundary conditions $\varphi_0=1$ and $\varphi_1=-1$, this allows us to recover the coefficients of $\varphi$ sequentially via a recurrence relation.

Exploiting the lower triangular Toeplitz structure of $T_{\varphi}$, we have
$$
e_s^\top A^{(\varphi)}p=e_s^\top P(T_\varphi)e_1= e_s^\top P(T_{\varphi,s})e_1,
$$
where $T_{\varphi,s}$ denotes the top principal $s\times s$ submatrix of $T_{\varphi}$. We decompose
$$
T_{\varphi,s}= \underbrace{\begin{bmatrix}
    \varphi_0\\
    \vdots&\ddots\\
    \varphi_{s-2}&\ddots&\ddots\\
    0&\varphi_{s-2}&\dots&\varphi_0
\end{bmatrix}}_{\widetilde{T}_{\varphi,s}}\quad+\quad\varphi_{s-1}\cdot \ e_s  e_1^\top,
$$
and observe that the two additive terms commute. 
In particular,  $$P(T_{\varphi,s})= P(\widetilde T_{\varphi,s}) + \varphi_{s-1}P'(\widetilde T_{\varphi,s})e_se_1^{\top} + \mathcal O((\varphi_{s-1}e_se_1^{\top})^2)=P(\widetilde T_{\varphi,s}) + \varphi_{s-1} P'(\widetilde T_{\varphi,s})e_se_1^\top,$$
where the last equality is due to the fact that $(e_se_1^{\top})^2=0$. Multiplying \eqref{eq:infinite-fixed-point} by $e_s^{\top}$ from the left gives
$$
m^{s-1}\varphi_{s-1} = e_s^\top P(\widetilde T_{\varphi,s})e_1 + \varphi_{s-1} e_s^\top P'(\widetilde T_{\varphi,s})e_s,
$$
from which we obtain
\begin{equation}\label{eq:forward}
\varphi_{s-1} =  \frac{e_s^\top P(\widetilde T_{\varphi,s})e_1}{m^{s-1}-e_s^\top P'(\widetilde T_{\varphi,s})e_s}. 
\end{equation}
The approach described above is equivalent to taking higher derivatives of the functional equations~\eqref{eq:poincare}, and evaluating at $z=0$; this leads to expressing a moment of $W$ of a certain order in terms of Bell polynomials evaluated at moments with lower orders, see~\cite[Section A.1]{hautphenne2020low}. On the other hand, writing $\varphi_s$ as a function of Toeplitz matrices comes with computational advantages.    
The vector $P(\widetilde T_{\varphi,s})e_1$  can be computed using Horner's method together
with fast Toeplitz matrix-vector multiplication, at a cost of $\mathcal{O}(d\,s\log s)$ per iteration. Note that this corresponds to computing the first $s$ coefficients of  the Taylor expansion of $P(\widetilde\varphi_{s-1}(z))$, where $\widetilde \varphi_{s-1}(z)= \sum_{j=0}^{s-2}\varphi_j z^j$. In our implementation, we use the recently proposed algorithm of \cite{kinoshita2024power}, which evaluates the coefficients of the composition with complexity $\mathcal O(\log(d)s\log s)$.  This approach significantly reduces the computational cost when the degree 
$d$ is large. Overall, computing the first $N$ coefficients of $\varphi(z)$ costs $\mathcal O(\log(d)N^2\log(N))$. 

In Section~\ref{sec:exp_newton_fixed} we will see that this approach can be slower than iterative algorithms and does not always provide the most accurate result. For this reason, we illustrate some alternatives in the next sections.

\subsection{A globally convergent functional fixed-point iteration}
The functional equation \eqref{eq:poincare} suggests considering the functional fixed-point iteration
\begin{equation}\label{eq:functional-fixed-point}
\varphi^{(k+1)}(z) = P\left(\varphi^{(k)}\left(\frac{z}{m}\right)\right),\quad k=0,1,2,\dots,
\end{equation}
starting from an initial function that satisfies the boundary conditions, for instance $\varphi^{(0)}(z)=1-z$.
A natural question is whether this iteration converges to the unique solution of \eqref{eq:poincare}. 

To prove global convergence of \eqref{eq:functional-fixed-point}, we work in a 
suitable metric space of functions. 
For $r>0$, let $H_2(r)$ denote the Hardy space  consisting of complex-valued functions that are holomorphic in the ball centered at the origin with radius $r>0$ and satisfy the boundary conditions of interest, that is,
$$
H_2(r):=\left \{
g := g(z)=\sum_{j=0}^\infty g_jz^j \,:\, g\text{ holomorphic for }|z|<r,\ g_0=1,\ g_1=-1
\right \}.
$$
It is well known that $H_2(r)$ equipped  with the $H_2(r)$-norm, defined as
$$
\|g\|_{H_2(r)}:=\sqrt{\sup_{s\in(0, r)}\frac{1}{2\pi }\int_0^{2\pi} \left\lvert g(se^{\mathbf i\theta})\right\rvert^2 \mathrm{d}\theta}= \sqrt{\sum_{j=0}^\infty r^{2j}|g_j|^2},
$$
 is a Hilbert space. 
 
 We then have the following convergence result.
\begin{theorem}\label{thm:convergence}
Let $\{\varphi^{(k)}\}_{k\in\mathbb N}$ be the sequence of functions generated by the functional iteration \eqref{eq:functional-fixed-point}  from an initial function $\varphi^{(0)}\in H_2(R)$, with $R>0$, and let $\varphi$ denote the solution of~\eqref{eq:poincare}. Then there exists $r\in(0, R)$ such that 
$$
 \lim_{k\to\infty} \norm{\varphi^{(k)} -\varphi}_{H_2(r)}=0.
$$
\end{theorem}

 \begin{proof}
Since $\varphi$ is entire, we have $\varphi\in H_2(r)$ for all $r>0$.
Choose $\rho>0$ such that $|P'(y)|\le m^2$ for all $y\in B(1, \rho)$, where $B(1, \rho)$ denotes the ball of center $1$ and radius $\rho$. Let $K_m := \sqrt{2m/(m-1)}$ and fix $0 < \delta < \rho / (2K_m)$. 

We now choose $r \in (0,\min\{1, R\})$ such that 
\begin{equation*}
    \varphi\left ( B \left ( 0, \frac{r}{m} \right ) \right ) \subseteq B\left ( 1, \frac{\rho}{2} \right ) \text{ and } \| \varphi^{(0)} - \varphi\|_{H_2(r)} < \delta.
\end{equation*}
Such a choice is always possible: for any fixed $\varphi^{(0)}\in H_2(R)$, the norm
$\|\varphi^{(0)}-\varphi\|_{H_2(r)}$ is nondecreasing in $r$ and tends to $0$ as $r\to0$.

We show by induction that 
\begin{equation}\label{eq:choice_r}
\varphi^{(k)} \in H_2(R) \text{ and } \|\varphi^{(k)} - \varphi\|_{H_2(r)} < \delta r^{2k}\qquad\textrm{for all $k \ge 0$,}
\end{equation}
 which proves the convergence of the iterates $\{\varphi^{(k)}\}_{k \in \mathbb{N}}$ to $\varphi$ in the $H_2(r)$-norm. We have chosen $r$ in such a way that the base step, with $k=0$, is true. 

Let us now consider the inductive step and assume~\eqref{eq:choice_r} holds for some $k\ge0$. Since $\varphi^{(k)} \in H_2(R)$, we have that $\varphi^{(k+1)} = P(\varphi^{(k)}(\cdot / m))$ is holomorphic in $B(0,R)$. Moreover, 
\begin{align*}
\varphi^{(k+1)}(0) &= P(\varphi^{(k)}(0))=P(1)=1,\\
\left(\varphi^{(k+1)}\right)'(0) &= \frac{1}{m}\cdot P'(\varphi^{(k)}(0))\cdot \left(\varphi^{(k)}\right)'(0)=\frac{1}{m} \cdot m \cdot (-1) =-1.
\end{align*}
This implies that $\varphi^{(k+1)}\in H_2(R)$.

Let $z$ be any point in $B(0, r)$; we can write 
\begin{align}
\left| \varphi(z) - \varphi^{(k+1)}(z) \right | & =  \left|P\left(\varphi\left(\frac zm\right)\right)-P\left(\varphi^{(k)}\left(\frac z m\right)\right)\right| \nonumber \\
& \le \max_{\zeta(z)} |P'\left(\zeta(z)\right)|\cdot \left|\varphi\left(\frac zm\right)- \varphi^{(k)}\left(\frac zm\right)\right|,\label{eq:lagrange}
\end{align}
where $\zeta(z)$ is a point on the segment connecting $\varphi\left(z/m\right)$ and $  \varphi^{(k)}\left(z/m\right)$, by the mean value theorem. 
Thanks to the choice of $r$ in~\eqref{eq:choice_r}, we have that $\varphi(z/m) \in B(1, \rho/2)$. Moreover, we have
\begin{equation*}
\left| \varphi\left (\frac{z}{m} \right ) - \varphi^{(k)}\left (\frac{z}{m}\right )\right| \le K_m \|\varphi - \varphi^{(k)}\|_{H_2(r)} \le \delta r^{2k} K_m < \rho/2,
\end{equation*}
where the first inequality\footnote{More precisely, we apply the lemma at page 36 in~\cite{duren1970theory} to the function $f(y) := \varphi(yr) - \varphi^{(k)}(yr)$, which satisfies $\|f\|_{H_2(1)} = \|\varphi - \varphi^{(k)}\|_{H_2(r)}$, at the point $\tilde y := y/m$, for $y \in B(0,1)$.} follows from~\cite[Chapter 3, page 36]{duren1970theory}, the second is the inductive hypothesis and the third follows from our choice of $\delta$. Therefore, both $\varphi(z/m)$ and $\varphi^{(k)}(z/m)$ lie in $B(1,\rho)$, and so does the segment that connects these two points. This means that $|P'(\zeta(z))| < m^2$ and using~\eqref{eq:lagrange} we get
\begin{align*}
\left\lVert \varphi - \varphi^{(k+1)} \right \rVert_{H_2(r)} & \le m^2 \left\lVert \varphi\left(\frac \cdot m\right)- \varphi^{(k)}\left(\frac \cdot m\right)\right\rVert_{H_2(r)} = m^2 \sqrt{\sum_{j=2}^{\infty}\frac{r^{2j}}{m^{2j}}|\varphi_{j}-\varphi^{(k)}_j|^2}\\
& \le r^2 \sqrt{\sum_{j=2}^\infty |\varphi_j - \varphi^{(k)}_j|^2} = r^2 \norm{\varphi\left(\cdot \right)- \varphi^{(k)}\left(\cdot \right)}_{H_2(r)} < r^{2k+2}\delta,
\end{align*}
where we used $r < 1$ in the inequalities, and this concludes the proof.
\end{proof}

Theorem~\ref{thm:convergence} shows that any initial function that is analytic in a neighbourhood of $0$ and satisfies the boundary conditions is a valid starting point for the iteration. However, some choices of $\varphi^{(0)}$ may lead to faster convergence and thus reduce the number of iterations required in practice.  In our numerical experiments, the choice $\varphi^{(0)}(z) = 1-z$ performs consistently well; see also the results reported in Section~\ref{sec:exp_newton_fixed}.

\subsection{Discretized fixed-point iteration}
Since dealing with an infinite number of Taylor coefficients all-at-once is numerically infeasible, we describe here how to reduce the functional fixed-point iteration to a finite-dimensional
problem.  By applying the inverse of $D$ to \eqref{eq:infinite-fixed-point} we get the fixed-point equation
\begin{equation}\label{eq:infinite-fixed-point2}
   \varphi = D^{-1} A^{(\varphi)} p,
\end{equation}
that corresponds, at the continuous level, to considering the equivalent equation $\varphi(z)= P(\varphi\left(\frac zm\right))$.
To obtain a practical computational procedure, we propose to truncate the infinite-dimensional fixed-point equation~\eqref{eq:infinite-fixed-point2} to a finite-dimensional system.
More precisely, for a given positive integer $N\ge d$, we denote by $D_N$ the truncation of $D$ to its leading $(N+1) \times (N+1)$ submatrix, by $A^{(\varphi)}_N$ the truncation of $A^{(\varphi)}$ to its first $N+1$ rows, and we consider the finite-dimensional vector iterative scheme
\begin{equation}\label{eq:fix-point-update}
     \varphi^{(k+1)} = D_N^{-1} A^{(\varphi^{(k)})}_N p \qquad \text{for } k\ge 0,
\end{equation}
where $\varphi^{(0)}(z)$ is chosen in the set of polynomials of degree bounded by $N$, so that $\varphi^{(0)}\in\mathbb R^{N+1}$. 

In addition, to preserve the boundary conditions, we perform the update~\eqref{eq:fix-point-update} only on the entries $\varphi^{(k+1)}_j$ for $j\ge 2$, and we set $\varphi_0^{(k+1)}=1$, and $\varphi_1^{(k+1)}=-1$. 

As a stopping criterion, we check whether the infinity norm of the relative residual 
$$
\mathrm{Res}(\varphi^{(k)}) := \max_{j\in\{0,\dots, N\}} \left|1 -\frac{m^j\varphi^{(k)}_j}{(A^{\varphi^{(k)}}p)_j}\right|,
$$
is below a given threshold $\varepsilon_\mathrm{tol}$. The procedure is summarized in Algorithm~\ref{alg:fixed-point}. 

\begin{algorithm}
  \caption{Fixed-point iteration for $\varphi(z)-P(\varphi\left(\frac zm\right))=0$}\label{alg:fixed-point}
  \begin{algorithmic}
    \Procedure{FixedPointPoincare}{$\varphi^{(0)}, p, N, \varepsilon_{\mathrm{tol}}$}
    \State $m\gets \sum_{j=1}^d j \cdot p_j$ 
    \State $D_N= \mathrm{diag}(1, m,\dots, m^{N})$
    \For{$k = 1, 2,3, \ldots $}
    \State $\widetilde{\varphi}^{(k)} \gets A_N^{\varphi^{(k-1)}}p$    
    \State $\varphi^{(k)}\gets D_N^{-1}\widetilde{\varphi}^{(k)} $
    \State $\varphi_0^{(k)}\gets 1,\quad \varphi_1^{(k)}\gets -1$
    \If{$\mathrm{Res(\varphi^{(k)})}\le \varepsilon_\mathrm{tol}$}
    \State \textbf{break}
    \EndIf
    \EndFor
    \State \Return $\varphi^{(k)}$
    \EndProcedure
  \end{algorithmic}
\end{algorithm}
Finally, we remark that the computational complexity of each iteration in Algorithm~\ref{alg:fixed-point} is dominated by the cost of evaluating $A^{(\varphi)}_N p$. This operation is equivalent to computing $P(T_{\varphi,N})e_1$, where $T_{\varphi, N}$ denotes the truncation of $T_{\varphi}$ to its first $N+1$ rows and columns, and $P(T_{\varphi,N}) = \sum_{j=0}^d p_j T_{\varphi,N}^j$. Similarly to section~\ref{sec:forward}, in place of evaluating $P(T_{\varphi,N})e_1$ we directly compute  the first $N+1$ coefficients in the Taylor expansion of $P(\varphi(z))$ via the algorithm in \cite{kinoshita2024power}, which requires $\mathcal O(\log(d)N\log N)$ arithmetic operations for each iteration.

\subsection{Newton's method}
As a second iterative method for solving \eqref{eq:poincare}, we consider
Newton's iteration applied to the functional equation
$$
\varphi(mz) -P(\varphi(z))=0,
$$
written in terms of the Taylor coefficients of $\varphi(z)$ and $P(z)$. The following lemma provides an explicit expression for the corresponding Jacobian.

\begin{lemma}\label{lem:jacobian}
Let $\varphi(z)$ be a formal power series, and let $A^{(\varphi)}$ be defined as in~\eqref{eq:toeplitz-krylov}. Define the vector $y \in \mathbb{R}^{d+1}$ by $y:= \begin{bmatrix}
    p_1&   2p_2&   \cdots &    d p_d &    0
\end{bmatrix}^\top.$
Then the Jacobian matrix $J(\varphi)$ , whose entries are given by $J(\varphi)_{i+1, s+1} = \frac{\partial}{\partial \varphi_s}\left(\varphi(mz) -P(\varphi(z))\right)_i$, can be written as
\begin{equation}\label{eq:J}
J(\varphi):= \begin{bmatrix}
    1\\ &m\\&&m^2\\&&&\ddots
\end{bmatrix}-\begin{bmatrix}
    \gamma_1\\
    \gamma_2&\gamma_1\\
    \gamma_3&\gamma_2&\gamma_1\\
    \vdots &\ddots&\ddots&\ddots
\end{bmatrix},
\end{equation}
where $\gamma_j$ denotes the inner product of the $j$th row of $A^{(\varphi)}$ with
the vector $y$.
\end{lemma}

\begin{proof}
Fix $s\in\{ 0,1,2,\ldots\}$. For each $i\ge 0$,
\begin{align*}
\frac{\partial}{\partial \varphi_s}\left(\varphi(mz) -P(\varphi(z))\right)_i &= m^{i}\delta_{is} - \sum_{j=1}^{d+1}\frac{\partial a_{ij}^{(\varphi)}}{\partial \varphi_s} p_{j-1}
= m^{i}\delta_{is} - \sum_{j=2}^{ d+1}\frac{\partial (e_i^\top T_{\varphi}^{j-1}e_1)}{\partial \varphi_s} p_{j-1}\\
& = m^{i}\delta_{is} - \sum_{j=2}^{d+1}e_i^\top\frac{\partial (T_{\varphi}^{j-1})}{\partial \varphi_s}e_1 p_{j-1},
\end{align*}
where $\delta_{is}$ denotes the Kronecker delta. Observe that adding $\varepsilon$ to $\varphi_s$  means adding $\varepsilon$ to the $s$-th sub-diagonal of $T_{\varphi}$, and 
$$
(T_{\varphi} + \varepsilon Z_s)^{j-1} - T_{\varphi}^{j-1}= \varepsilon(j-1)T_{\varphi}^{j-2}Z_s+ \mathcal O(\varepsilon ^2), \text{ where } Z_s= \begin{bmatrix}
&\\
    &\\
    1\\
    &1\\
    &&\ddots
\end{bmatrix}
$$
has zeros in the first $s$ rows. This implies 
$$\frac{\partial}{\partial \varphi_s}\left(\varphi(mz) -P(\varphi(z))\right)_i = m^i\delta_{is} - \sum_{j=2}^{d+1}e_{i+1}^\top T_{\varphi}^{j-2}e_{s+1} (j-1)p_{j-1}.$$
The quantity $e_{i+1}^\top T_{\varphi}^{j-2}e_{s+1}$ represents the coefficient of $z^{i-s}$ in the series expansion of $\varphi(z)^{j-2}$, i.e., it is $0$ when $s>i$, and corresponds to the entry $(i-s+1,j-1)$ of $A^{(\varphi)}$ otherwise. Therefore, for $i\ge s$, the quantity $\left(\frac{\partial P(\varphi(z))}{\partial \varphi_s}\right)_i$ is equal to the product between the $(i-s+1)$-th row of $A^{(\varphi)}$ and the vector $y$. Note that the $(i+1,s+1)$-th entry of the Jacobian matrix $J(\varphi)$ is $m^{i}\delta_{is}-\frac{\partial \left(P(\varphi(z))\right)_i}{\partial\varphi_s} $, and this implies~\eqref{eq:J}.
\end{proof}

We remark that the first row of $T_{\varphi}^j$ is $e_1^\top$ for all $j$, so the first row of $A^{(\varphi)}$ is $[1, 1, \ldots, 1]$, therefore $\gamma_1 = p_1 + 2p_2 + 3p_3 + \ldots + dp_d = m $. This means that the Jacobian matrix is rank-deficient, as its first two rows are linearly dependent. On the other hand, the first two rows of $J(\varphi)$ only involve the first two Taylor coefficients of $\varphi^{(k)}(z)$, which are fixed by the boundary conditions of~\eqref{eq:poincare}. To overcome this technical difficulty, we rewrite the updating process for all coefficients but $\varphi^{(k)}_0,\varphi^{(k)}_1$. Let us introduce 
\begin{align*}
\widetilde{\varphi}^{(k)}&=\begin{bmatrix}
    \varphi^{(k)}_2\\
    \varphi^{(k)}_3\\
    \vdots
\end{bmatrix},\quad 
\widetilde D = \begin{bmatrix}
    m^2\\
    &m^3\\
    &&\ddots
\end{bmatrix},
\quad w^{(k)}= \begin{bmatrix}
(A^{\varphi^{(k)}}p)_2   \\
(A^{\varphi^{(k)}}p)_3\\
\vdots
\end{bmatrix},\\
\widetilde J(\varphi)&= 
\begin{bmatrix}
    m^2\\ &m^3\\&&m^4\\&&&\ddots
\end{bmatrix}-\begin{bmatrix}
    \gamma_1\\
    \gamma_2&\gamma_1\\
    \gamma_3&\gamma_2&\gamma_1\\
    \vdots &\ddots&\ddots&\ddots
\end{bmatrix}.
\end{align*}
Then, the Newton iteration reads
\begin{equation}\label{eq:newton-scheme}
\begin{cases}\widetilde{\varphi}^{(k+1)} = \widetilde{\varphi}^{(k)} - \widetilde{J}(\varphi^{(k)})^{-1}(\widetilde D\widetilde{\varphi}^{(k)}-w^{(k)})\\
\varphi^{(k+1)}_0=1,\ \varphi^{(k+1)}_1=-1
\end{cases}, \qquad k\ge 0.
\end{equation}
Once again, the practical implementation of  Newton's method in this context is obtained by truncating the semi-infinite
objects in~\eqref{eq:newton-scheme} to size $N$. The resulting procedure is reported in
Algorithm~\ref{alg:newton}.

\begin{algorithm}
  \caption{Newton's method for $\varphi(mz)-P(\varphi(z))=0$}\label{alg:newton}
  \begin{algorithmic}
    \Procedure{NewtonPoincare}{$\varphi^{(0)}, p, N, \varepsilon_{\mathrm{tol}}$}
    \State $m\gets \sum_{j=1}^d j \cdot p_j$ 
    \State $\widetilde{D}_N= \mathrm{diag}(m^2, m^3,\dots, m^{N})$
    \State $y=[p_1, 2p_2,\dots, dp_d,0, \dots, 0]^{\top}\in\mathbb R^{N+1}$
    \For{$k = 1, 2,3, \ldots $}
    \State $\gamma = A_N^{\varphi^{(k-1)}} y$ 
    \State $w^{(k)} \gets A_N^{\varphi^{(k-1)}}p$    
    \State $w^{(k)} \gets [w^{(k)}_2,\dots,w^{(k)}_N]^{\top}$
    \State $\widetilde{J}_N(\varphi^{(k)})\gets \left[\begin{smallmatrix}
    m^2\\ &m^3\\&&\ddots\\&&&m^N
\end{smallmatrix}\right]-\left[\begin{smallmatrix}
    \gamma_1\\
    \gamma_2&\gamma_1\\
    \vdots&\ddots&\ddots\\
    \gamma_{N-1} &\dots&\dots&\gamma_1
\end{smallmatrix}\right]$
\State $\widetilde{\varphi}^{(k-1)}\gets [\widetilde{\varphi}^{(k-1)}_2,\dots, \widetilde{\varphi}^{(k-1)}_N]^{\top}$
\State $[\varphi_2^{(k)},\dots, \varphi_N^{(k)}]^{\top} \gets \widetilde{\varphi}_{(k-1)}-\widetilde{J}_N(\varphi^{(k-1)})^{-1}(\widetilde D_N\widetilde\varphi^{(k-1)}-w^{(k)})$
    \State $\varphi_0^{(k)}\gets 1,\quad \varphi_1^{(k)}\gets -1$
    \If{$\mathrm{Res(\varphi^{(k)})}\le \varepsilon_{\mathrm{tol}}$}
    \State \textbf{break}
    \EndIf
    \EndFor
    \State \Return $\varphi^{(k)}$
    \EndProcedure
  \end{algorithmic}
\end{algorithm}

Each iteration of Algorithm~\ref{alg:newton} requires computing the first $N+1$ coefficients of the composition of two polynomials (using the algorithm of \cite{kinoshita2024power}), as well as solving a lower triangular linear system of size $N-1$. Since all remaining operations have lower complexity, the overall cost
of a single iteration is $\mathcal O(N^2+\log(d)\,N\log N)$.

\begin{remark}\label{rem:global-conv}
Proceeding as in the proof of Lemma~\ref{lem:jacobian}, we find that the Jacobian of the finite-dimensional fixed-point iteration~\eqref{eq:fix-point-update} is given by
$$
\begin{bmatrix}
    m^2\\ &m^3\\&&\ddots\\&&&m^N
\end{bmatrix}^{-1}\begin{bmatrix}
    \gamma_1\\
    \gamma_2&\gamma_1\\
        \vdots &\ddots&\ddots&\\
    \gamma_{N-1}&\dots&\dots&\gamma_1
\end{bmatrix}= \begin{bmatrix}
    m^{-1}\\
    m^{-3}\gamma_2&m^{-2}\\
        \vdots &\ddots&\ddots&\\
    m^{-N}\gamma_{N-1}&\dots&m^{-N}\gamma_2&m^{1-N}
\end{bmatrix}.
$$
This matrix is lower triangular, with diagonal entries bounded by $1/m<1$, so its
spectral radius equals $1/m$. This shows that the truncated fixed-point map is a
contraction, and therefore that Algorithm~\ref{alg:fixed-point} converges for any
initial vector satisfying $\varphi^{(0)}_0=1$ and $\varphi^{(0)}_1=-1$.
\end{remark}

\subsection{Numerical examples}\label{sec:exp_newton_fixed}
In this section, we compare the performances of Algorithm~\ref{alg:fixed-point}, Algorithm~\ref{alg:newton}, and the forward substitution formula~\eqref{eq:forward} on two numerical case studies.
\begin{example}
In this first numerical test we consider the offspring distribution
$$
P(z)=\frac{1}{10}(1+z+z^2+\ldots+z^9)
$$
and we study the convergence rate of the two proposed methods. For both methods, we set $N=100$, corresponding to a Taylor approximation of $\varphi(z)$ of degree $100$, and we use the initial function $\varphi^{(0)}(z)=1-z$. 

Figure~\ref{fig:test1_paper} shows the evolution of the relative residual
$\mathrm{Res}(\varphi^{(k)})$ throughout the iterations. Newton's method converges superlinearly and reaches machine precision after only four iterations. 
In contrast, the fixed-point iteration exhibits linear convergence and requires $27$ iterations to achieve a comparable level of accuracy.

 After $27$ iterations, the relative infinity-norm difference between the two approximate solutions, 
$$
\max_{j\in\{0,\dots,N\}}\left|1-\frac{\varphi^{(\mathrm{Newton})}_j}{\varphi_j^{(\mathrm{Fixed})}}\right|,
$$
is of the order of $10^{-15}$. This indicates that, up to numerical precision, both
methods converge to the same fixed point.
\end{example}
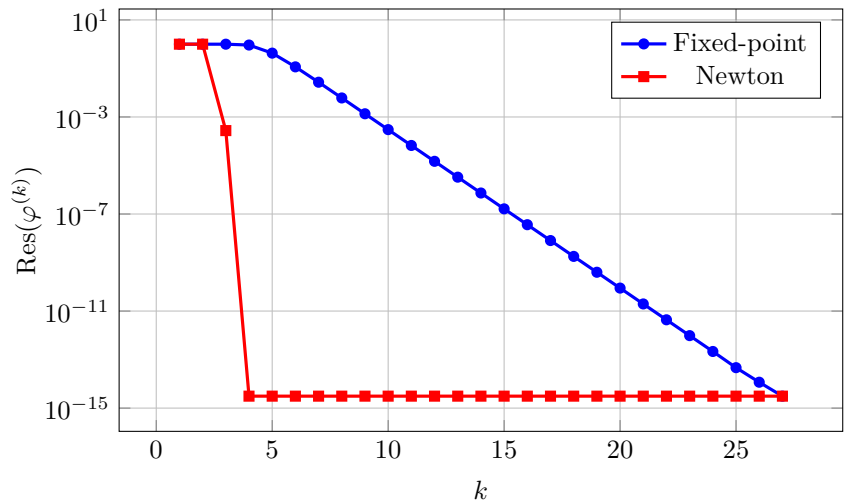
\begin{figure}[h]
\centering
\begin{tikzpicture}
  \begin{semilogyaxis}[
    xlabel={$k$},
    ylabel={$\mathrm{Res}(\varphi^{(k)})$},
    legend pos=north east,
    grid=major,
    width=.7\textwidth,
    height=0.45\textwidth
  ]  
  \addplot[mark=*, very thick, blue, mark size=1.5pt] table [x index=0, y index=1, col sep=space] {test1_err2.dat};
  \addlegendentry{Fixed-point}
  \addplot[mark=square*, very thick, red, mark size=1.5pt] table [x index=0, y index=2, col sep=space] {test1_err2.dat};
  \addlegendentry{Newton}
  \end{semilogyaxis}
\end{tikzpicture}
\caption{Relative residual $\mathrm{Res}(\varphi^{(k)})$ along the iterations of the fixed-point method and Newton's method when solving \eqref{eq:poincare} for $P(z)=\frac{1}{10}(1+z+z^2+\dots+z^9)$. }
\label{fig:test1_paper}
\end{figure}

\begin{example}
We now investigate the influence of the degree $d$ of $P(z)$ and of the mean offspring number $m$ on the performances of the two iterative methods, and of the direct solver based on the formula~\eqref{eq:forward}. The experiment is run for $d\in\{5,10,15,20,50\}$ and $m\in\{1.1,1.25,2,3\}$. For each configuration $(d,m)$ we generate $100$ offspring distributions at random and we execute both Algorithm~\ref{alg:fixed-point} and Algorithm~\ref{alg:newton} with input arguments $N=100$, $\varepsilon_{\mathrm{tol}}=10^{-8}$, and $\varphi^{(0)}(z)=1-z$.  

The average execution times, numbers of iterations, and final relative residuals are reported in Table~\ref{tab:exp2}. We observe that, as $m$ approaches $1$, the fixed-point method requires significantly more iterations than Newton's method. This is expected as the spectral radius of the Jacobian of the fixed-point iteration is $\frac{1}{m}$; see Remark~\ref{rem:global-conv}. As a result, Newton's method is faster than the fixed-point iteration by a factor between $10$ and $30$ when $m\in\{1.1,1.25\}$. For larger values of $m$, the execution times of the two methods become comparable.  

The execution time of the forward substitution approach is not affected by the parameter $m$; the values reported in Table~\ref{tab:exp2} show that this approach is always slower than Newton's method, and is comparable to or faster than the fixed-point iteration only when the latter requires $\mathcal O(N)$ iterations, i.e., when $m\in\{1.1, 1.25\}$. 

The degree $d$ has only a minor impact on the computational costs, in agreement with the complexity analysis discussed earlier. However, the latter plays a major role when we look at the accuracy of the forward substitution method. The latter deteriorates as $d$ increases, providing residuals of the order $10^{-4}$ when $d=50$. Instead, the accuracy of the iterative methods are less affected by $d$. The fixed-point iteration
typically stops with a residual close to the prescribed tolerance $10^{-8}$, reflecting its linear convergence. Newton’s method often attains residuals well below
$\varepsilon_{\mathrm{tol}}$, empirically exhibiting superlinear convergence.
\end{example}
\begin{table}
\caption{Execution time (in seconds), number of iterations, and final residual for the fixed-point iteration \eqref{eq:fix-point-update} (left), Newton's method \eqref{eq:newton-scheme} (right), and the forward substitution \eqref{eq:forward}, when solving \eqref{eq:poincare}. Several values of the maximum number of offspring $d$ and of the mean offspring number $m$ are considered. For each pair $(d,m)$, the offspring distribution $P(z)$ is generated at random, and the reported results are averaged over $100$ runs.}\label{tab:exp2}
\centering
\resizebox{\textwidth}{!}{
\begin{tabular}{|c|c|c|c|c|c|}
\hline
\multicolumn{6}{|c|}{\textbf{Fixed-point}}\\
\hline
\multicolumn{6}{|c|}{Time}\\
\hline
   \diagbox[width=\dimexpr \textwidth/20+2\tabcolsep\relax, height=.75cm]{d}{m}  & $1.1$ & $1.25$ & $1.5$ & $2$ & $3$ \\
   \hline
$5$ & $0.111$ & $0.057$ & $0.033$ & $0.020$ & $0.013$ \\
$10$ & $0.145$ & $0.065$ & $0.037$ & $0.023$ & $0.015$ \\
$15$ & $0.150$ & $0.068$ & $0.039$ & $0.024$ & $0.016$ \\
$20$ & $0.172$ & $0.077$ & $0.044$ & $0.027$ & $0.018$ \\
$50$ & $0.254$ & $0.114$ & $0.065$ & $0.040$ & $0.026$ \\

\hline
\multicolumn{6}{|c|}{Iterations}\\
\hline
\diagbox[width=\dimexpr \textwidth/20+2\tabcolsep\relax, height=.75cm]{d}{m}& $1.1$ & $1.25$ & $1.5$ & $2$ & $3$ \\
\hline
$5$ & $220.0$ & $99.8$ & $58.0$ & $36.0$ & $24.3$ \\
$10$ & $219.0$ & $99.0$ & $57.0$ & $35.0$ & $24.0$ \\
$15$ & $219.0$ & $99.0$ & $57.0$ & $35.0$ & $23.4$ \\
$20$ & $219.0$ & $99.0$ & $57.0$ & $35.0$ & $23.0$ \\
$50$ & $219.0$ & $99.0$ & $57.0$ & $35.0$ & $23.0$ \\
 \hline 
\multicolumn{6}{|c|}{Residual}\\
\hline
\diagbox[width=\dimexpr \textwidth/20+2\tabcolsep\relax, height=.75cm]{d}{m}& $1.1$ & $1.25$ & $1.5$ & $2$ & $3$ \\
\hline
$5$ & $9\cdot 10^{-9}$ & $8\cdot 10^{-9}$ & $8 \cdot 10^{-9}$ & $7\cdot 10^{-9}$ & $7\cdot 10^{-9}$ \\
$10$ & $1\cdot 10^{-8}$ & $9\cdot 10^{-9}$ & $9\cdot 10^{-9}$ & $9\cdot 10^{-9}$ & $4\cdot 10^{-9}$ \\
$15$ & $1 \cdot 10^{-8}$ & $9\cdot 10^{-9}$ & $8\cdot 10^{-9}$ & $8\cdot 10^{-9}$ & $7\cdot 10^{-9}$ \\
$20$ & $1\cdot 10^{-8}$ & $8\cdot 10^{-9}$ & $8\cdot 10^{-9}$ & $7\cdot 10^{-9}$ & $9\cdot 10^{-9}$ \\
$50$ & $9\cdot 10^{-9}$ & $8\cdot 10^{-9}$ & $7\cdot 10^{-9}$ & $6\cdot 10^{-9}$ & $7\cdot 10^{-9}$ \\
  \hline
\end{tabular}~
\begin{tabular}{|c|c|c|c|c|c|}
\hline
\multicolumn{6}{|c|}{\textbf{Newton}}\\
\hline
\multicolumn{6}{|c|}{Time}\\
\hline
   \diagbox[width=\dimexpr \textwidth/20+2\tabcolsep\relax, height=.75cm]{d}{m}  & $1.1$ & $1.25$ & $1.5$ & $2$ & $3$ \\
   \hline
$5$ & $0.006$ & $0.006$ & $0.007$ & $0.006$ & $0.005$ \\
$10$ & $0.008$ & $0.008$ & $0.009$ & $0.007$ & $0.007$ \\
$15$ & $0.008$ & $0.008$ & $0.008$ & $0.008$ & $0.008$ \\
$20$ & $0.009$ & $0.009$ & $0.009$ & $0.009$ & $0.009$ \\
$50$ & $0.014$ & $0.014$ & $0.012$ & $0.012$ & $0.013$ \\

\hline \multicolumn{6}{|c|}{Iterations}\\
\hline
\diagbox[width=\dimexpr \textwidth/20+2\tabcolsep\relax, height=.75cm]{d}{m}& $1.1$ & $1.25$ & $1.5$ & $2$ & $3$ \\
\hline
$5$ & $6.0$ & $6.0$ & $6.0$ & $5.6$ & $5.0$ \\
$10$ & $6.0$ & $6.0$ & $6.0$ & $5.0$ & $5.0$ \\
$15$ & $6.0$ & $6.0$ & $5.0$ & $5.0$ & $5.0$ \\
$20$ & $6.0$ & $6.0$ & $5.0$ & $5.0$ & $5.0$ \\
$50$ & $6.0$ & $6.0$ & $5.0$ & $5.0$ & $5.0$ \\
  \hline 
\multicolumn{6}{|c|}{Residual}\\
\hline
\diagbox[width=\dimexpr \textwidth/20+2\tabcolsep\relax, height=.75cm]{d}{m}& $1.1$ & $1.25$ & $1.5$ & $2$ & $3$ \\
\hline
$5$ & $2\cdot 10^{-12}$ & $1 \cdot 10^{-14}$ & $1\cdot 10^{-14}$ & $2 \cdot 10^{-09}$ & $2\cdot 10^{-12}$ \\
$10$ & $1 \cdot 10^{-13}$ & $1\cdot 10^{-14}$ & $3 \cdot 10^{-10}$ & $4\cdot 10^{-14}$ & $7\cdot 10^{-15}$ \\
$15$ & $5\cdot 10^{-14}$ & $1 \cdot 10^{-14}$ & $2 \cdot 10^{-09}$ & $5\cdot 10^{-15}$ & $6\cdot 10^{-15}$ \\
$20$ & $4\cdot 10^{-14}$ & $1 \cdot 10^{-14}$ & $8\cdot 10^{-10}$ & $5\cdot 10^{-15}$ & $4\cdot 10^{-15}$ \\
$50$ & $4\cdot 10^{-10}$ & $3\cdot 10^{-12}$ & $2\cdot 10^{-10}$ & $8\cdot 10^{-13}$ & $1\cdot 10^{-13}$ \\
  \hline
\end{tabular}
}
\vspace{.2cm}

\resizebox{0.5\textwidth}{!}{
\begin{tabular}{|c|c|c|c|c|c|}
\hline
\multicolumn{6}{|c|}{\textbf{Forward}}\\
\hline
\multicolumn{6}{|c|}{Time}\\
\hline
   \diagbox[width=\dimexpr \textwidth/20+2\tabcolsep\relax, height=.75cm]{d}{m}  & $1.1$ & $1.25$ & $1.5$ & $2$ & $3$ \\
   \hline
$5$ & $0.046$ & $0.049$ & $0.049$ & $0.049$ & $0.049$ \\
$10$ & $0.061$ & $0.061$ & $0.061$ & $0.061$ & $0.061$ \\
$15$ & $0.066$ & $0.066$ & $0.066$ & $0.066$ & $0.067$ \\
$20$ & $0.079$ & $0.079$ & $0.079$ & $0.079$ & $0.079$ \\
$50$ & $0.136$ & $0.136$ & $0.136$ & $0.136$ & $0.136$ \\
\hline \multicolumn{6}{|c|}{Residual}\\
\hline
\diagbox[width=\dimexpr \textwidth/20+2\tabcolsep\relax, height=.75cm]{d}{m}& $1.1$ & $1.25$ & $1.5$ & $2$ & $3$ \\
\hline
$5$ & $3\cdot 10^{-15}$ & $3\cdot 10^{-15}$ & $3\cdot 10^{-15}$ & $3\cdot 10^{-15}$ & $3\cdot 10^{-15}$ \\
$10$ & $4\cdot 10^{-15}$ & $3\cdot 10^{-15}$ & $3\cdot 10^{-15}$ & $3\cdot 10^{-15}$ & $3\cdot 10^{-15}$ \\
$15$ & $4\cdot 10^{-14}$ & $4\cdot 10^{-14}$ & $3\cdot 10^{-14}$ & $3\cdot 10^{-14}$ & $2\cdot 10^{-14}$ \\
$20$ & $1\cdot 10^{-12}$ & $9\cdot 10^{-13}$ & $7\cdot 10^{-13}$ & $5\cdot 10^{-13}$ & $4\cdot 10^{-13}$ \\
$50$ & $2\cdot 10^{-4}$ & $2\cdot 10^{-4}$ & $2\cdot 10^{-4}$ & $1\cdot 10^{-4}$ & $8\cdot 10^{-5}$ \\
  \hline
\end{tabular}}
\end{table}
\section{Recovering the density of $W$}\label{sec:momentmatching}
Once the first $N+1$ coefficients $\varphi_0,\varphi_1,\ldots,\varphi_{N}$ of the Taylor expansion of the Laplace-Stieltjes transform $\varphi(z) =\mathbb{E}[\exp(-zW)]$ have been approximated, we can address the problem of recovering of the density function $f(x)$ of $W$. Since the Laplace-Stieltjes transform is entire, the sequence of moments uniquely determines the density. A natural approach is therefore to approximate $f(x)$ by imposing a moment-matching condition on a suitable vector space of functions. More specifically, we consider an approximant  expressed as a linear combination of generalized orthogonal polynomials, a standard technique for reconstructing a density from its moments; see, for instance, \cite{Provost2012}.

\paragraph{Choice of the basis functions.}To select a suitable family of polynomials, we rely on specific features of the density $f(x)$, which can either be inferred \textit{a priori} or estimated empirically. It is well known that the distribution of $W$ has a point mass at $0$, corresponding to the extinction probability $q$, and is absolutely continuous on the interval $(0, \infty)$; see~\cite[Theorem 8.3]{Harris1963}. The value of $q$ can be computed as the smallest real solution of the equation $P(z)=z$; see, for instance,~\cite[Chapter~5, Theorem~1]{Athreya2012}. The absolutely continuous part of the distribution 
has support on $(0,\infty)$ and, in a right neighborhood of $0$, its density behaves like $x^{\alpha}$, where 
\begin{equation}\label{eq:a}
\alpha:= -\frac{\log P'(q)}{ \log m} - 1;
\end{equation}see~\cite[Theorem 1]{Dubuc1971}. Moreover, the right tail of the distribution decays at least exponentially; see, e.g.,~\cite{Fernley2024}.

In view of the above results, we propose to approximate the density function of $W$ by 
\begin{equation}\label{eq:fN}
	f(x) \approx f_S(x) := q\, \delta_0(x)+\sum_{j=0}^S c_j L_j^{(\alpha)}(\beta x) \, (\beta x)^\alpha e^{-\beta x},
\end{equation}
where $\delta_0(x)$ denotes the Dirac delta function, $\beta>0$ is a parameter,  $c_0, c_1, \ldots, c_S \in \mathbb{R}$ are coefficients to be determined, and $L_j^{(\alpha)}(x)$ are the generalized Laguerre polynomials with parameter $\alpha$, defined as follows. 
\begin{definition}
	The generalized Laguerre polynomials with parameter $\alpha$ are defined recursively by
	\begin{equation*}
L_j^{(\alpha)}(x) =
\begin{cases}
1, & j=0,\\[0.3em]
1+\alpha-x, & j=1,\\[0.3em]
\displaystyle
\frac{2j+\alpha-1-x}{j} L_{j-1}^{(\alpha)}(x)
-\frac{j+\alpha-1}{j} L_{j-2}^{(\alpha)}(x),
& j\ge 2.
\end{cases}
\end{equation*}
\end{definition}
The family $\{L_j^{(\alpha)}(\beta x)\}_{j\ge0}$ forms a system of orthogonal polynomials on $(0, +\infty)$ with respect to the weight $w(x) := (\beta x)^\alpha e^{-\beta x}$. Moreover, the inner products of $L_j^{(\alpha)}(\beta x)$ with the monomials $x^i$ (with respect to the weight $w(x)$) can be computed in closed form.

 The moment-matching approach consists in choosing the coefficients $c_0, \ldots, c_S$ in such a way that the first $N+1$ moments of $f_S(x)$ match the first $N+1$ moments of $W$. The moments $m_i = \mathbb{E}[W^i]$ are related to the
coefficients of the Taylor expansion of $\varphi$ via
\begin{equation}\label{eq:phimoments}
    \varphi_i = (-1)^i \frac{m_i}{i!} \Longleftrightarrow m_i = (-1)^i i! \varphi_i, \qquad i=0,1,\ldots,N.
\end{equation}
Substituting the expression of $f_S(x)$ into the integral defining
$\mathbb{E}[W^i]$, the moment-matching conditions can be written as the linear
system
\begin{equation}\label{eq:linsys}
    m_i
    = \sum_{j=0}^S c_j
      \int_0^\infty x^{i+\alpha}\beta^\alpha
      L_j^{(\alpha)}(\beta x)e^{-\beta x}\,\mathrm{d}x
      + \begin{cases}
          q, & i=0,\\
          0, & \text{otherwise},
        \end{cases}
\end{equation}
for $i=0,1,\ldots,N$. The number $S+1$ of unknowns coefficients may differ from the number $N+1$ of computed moments; in this case, we can view~\eqref{eq:linsys} as a least-squares problem instead. 
Let $B\in\mathbb{R}^{(N+1)\times(S+1)}$ denote the matrix with entries
\begin{equation*}
    b_{i,j}
    := \int_0^\infty x^{i+\alpha}\beta^\alpha
       L_j^{(\alpha)}(\beta x)e^{-\beta x}\,\mathrm{d}x, \quad i=0,\ldots,N, \;j=0,\ldots,S.
\end{equation*}
The matrix $B$, associated with the linear system~\eqref{eq:linsys}, has a
peculiar structure, which we characterize in the next theorem.

\begin{theorem}\label{thm:B}
    The matrix $B$ can be factorized as
    \begin{equation}\label{eq:B}
    \begin{small}
    B = \begin{bmatrix} \beta^{-1}\Gamma(\alpha+1) & & & \\ & \beta^{-2}\Gamma(\alpha+2) & & \\ & & \ddots & \\ & & & \beta^{-N-1}\Gamma(\alpha+N+1) \end{bmatrix} M \begin{bmatrix} 1 & & & \\ & -1 & & \\ & & \ddots & \\ & & & (-1)^{S} \end{bmatrix},
    \end{small}
    \end{equation}
    where $M\in\mathbb R^{(N+1)\times (S+1)}$ is the (rectangular) lower triangular matrix defined by \[
M_{i,j} =
\begin{cases}
\binom{i}{j}, & i \ge j,\\
0,            & i < j.
\end{cases}
\]
\end{theorem}

\begin{proof}
With respect to the $L^2(0, \infty)$ inner product weighted by $(\beta x)^\alpha e^{-\beta x}$, 
the generalized Laguerre polynomial $L_j^{(\alpha)}(\beta x)$ is orthogonal to all polynomials of degree at most $j-1$. Therefore, $b_{i,j} = 0$ whenever $i < j$, so the matrix $B$ is lower triangular. 

Now let us prove that 
\begin{equation}\label{eq:induction}
b_{i,j} = (-1)^j \cdot \beta^{-i-1}\cdot \Gamma(\alpha+i+1) \cdot \binom{i}{j}
\end{equation}for all $i \ge j$, by induction on $j$. For $j=0$, we have
\begin{equation*}
    b_{i,0} = \beta^\alpha \int_{0}^\infty x^{i+\alpha} e^{-\beta x} \mathrm{d}x = \beta^{-i-1}\Gamma(\alpha+i+1) \text{ for all } i \ge 0.
\end{equation*}
For $j=1$, using $L_1^{(\alpha)}(x)=1+\alpha-x$, we obtain
\begin{align*}
b_{i,1}
&= \beta^\alpha \int_0^\infty x^{i+\alpha}(1+\alpha-\beta x)e^{-\beta x}\,\mathrm{d}x \\
&= (1+\alpha)\beta^{-i-1}\Gamma(\alpha+i+1)
   -\beta^{-i-1}\Gamma(\alpha+i+2) \\
&= -(i+1)\beta^{-i-1}\Gamma(\alpha+i+1),
\qquad i\ge1,
\end{align*}
which agrees with~\eqref{eq:induction}.

Now assume that \eqref{eq:induction} holds for all values of $\widetilde j \le j-1$ and all $i$. Using the recurrence relation for the generalized Laguerre polynomials, 
we write
\begin{align*}
b_{i,j}
&= \beta^\alpha \int_0^\infty x^{i+\alpha}
\left(
\frac{2j+\alpha-1-\beta x}{j}L_{j-1}^{(\alpha)}(\beta x)
-\frac{j+\alpha-1}{j}L_{j-2}^{(\alpha)}(\beta x)
\right)e^{-\beta x}\,\mathrm{d}x \\
&= \frac{2j+\alpha-1}{j}b_{i,j-1}
   -\frac{\beta}{j}b_{i+1,j-1}
   -\frac{j+\alpha-1}{j}b_{i,j-2}.
\end{align*}
Substituting the induction hypothesis yields
$
b_{i,j}
= (-1)^j \beta^{-i-1}\Gamma(\alpha+i+1)\,X,
$
where
\begin{align*}
X
&= -\frac{2j+\alpha-1}{j}\binom{i}{j-1}
   +\frac{\alpha+i+1}{j}\binom{i+1}{j-1}
   -\frac{j+\alpha-1}{j}\binom{i}{j-2}.
\end{align*}
A direct computation using standard binomial identities shows that $X=\binom{i}{j}$
for all $i\ge j$, which completes the inductive step.

\end{proof}

Let $c = \begin{bmatrix} c_0 & c_1 & \cdots & c_S \end{bmatrix}^\top$ denote the vector of unknown coefficients. Exploiting the special structure of the marix $B$ given in Theorem~\ref{thm:B}, we can rewrite the linear system~\eqref{eq:linsys} in the equivalent form

\begin{equation}\label{eq:linsysequiv}
M\,y = b,
\qquad
c_j = (-1)^j y_j,\quad j=0,\ldots,S,
\end{equation}
where $M$ is the lower triangular matrix defined in
Theorem~\ref{thm:B}, and the right-hand side $b\in\mathbb{R}^{N+1}$ is given by
\begin{equation}
b :=
\begin{bmatrix}\label{bb}
\dfrac{\beta\,\Gamma(1)}{\Gamma(\alpha+1)}(1-q) &
\dfrac{\beta^2\,\Gamma(2)}{\Gamma(\alpha+2)} &
\dfrac{\beta^3\,\Gamma(3)}{\Gamma(\alpha+3)}\,|\varphi_2| &
\dots &
\dfrac{\beta^{N+1}\Gamma(N+1)}{\Gamma(\alpha+N+1)}\,|\varphi_N|
\end{bmatrix}^{\top}.
\end{equation}

\begin{remark}
    Computing all the ratios of Gamma functions appearing in~\eqref{eq:linsysequiv} can be done in an efficient and robust way exploiting the relation
    \begin{equation*}
        \frac{\Gamma(j+\alpha)}{\Gamma(j)} = \Gamma(\alpha) \prod_{i=1}^j \frac{i}{i+\alpha}.
    \end{equation*}
    In particular, computing $b$ defined in \eqref{bb} requires only a single evaluation of $\Gamma(\alpha)$ and $\mathcal O(N)$ arithmetic operations.
\end{remark}



\begin{remark}
The matrix $M$ appearing in~\eqref{eq:B} coincides with the lower triangular factor in the Cholesky factorization of a Pascal matrix $\mathcal P$, whose entries are given by $\mathcal{P}_{i,j} = \binom{i+j}{j}$ for all $i,j \ge 0$; see, for instance,~\cite{Edelman2004}. Theorem~\ref{thm:B} thus reveals an elegant connection between moment-based density
reconstruction using generalized Laguerre polynomials and classical structures
associated with Pascal matrices.
\end{remark}

\paragraph{Mitigating the conditioning.}Pascal matrices $\mathcal P_n$ of size $n \times n$ have very large (nonnegative) entries and quickly become ill-conditioned as $n$ increases. In particular, their condition number grows asymptotically like $16^n / (n \pi)$; see~\cite[Section~28.4]{Higham2002}. This ill-conditioning extends to the matrix $M$ that appears in~\eqref{eq:B}. This means, in particular, that small errors in the computation of the coefficients of $\varphi(z)$ (and therefore in the computation of the moments of $W$) could result in large errors in the approximation of the coefficients $c_j$ in~\eqref{eq:fN}. 

To mitigate this problem, we adopt the following strategy:
\begin{enumerate}
\item For the matrix $M \in \mathbb{R}^{(N+1) \times (S+1)}$ and vector $b \in 
\mathbb{R}^{N+1}$ in~\eqref{eq:linsysequiv} and \eqref{bb}, we consider a rescaled version of the system in which the matrix $M$ is divided by the maximum entry in each row. More specifically, we let $\widetilde M := D_M^{-1} M$ and $\widetilde b := D_M^{-1} b$ denote the rescaled coefficient matrix and right-hand side, where $D_M$ is the diagonal matrix with entries $(D_M)_{j,j} = \max_i M_{j,i}$. 

\item We set $S+1 = \lfloor (N+1)/2\rfloor$, that is, we approximate the density of $W$ using an expansion~\eqref{eq:fN} with roughly half as many basis functions as available
moments. The coefficients are then obtained by solving the least-squares problem
\begin{equation}\label{eq:LS}
\min_y \| \widetilde M y - \widetilde b\| = \min_y \| D_M^{-1}My - D_M^{-1}b\|, \qquad c = \mathrm{diag}(1,-1,\ldots,(-1)^S) \cdot y.
\end{equation}
This choice is motivated by the fact that the resulting matrix $\widetilde M$ is substantially better conditioned, as illustrated in Figure~\ref{fig:pascal}.
\end{enumerate}
Although the condition number of $\widetilde M$ still grows rapidly with the matrix size, satisfactory approximations of the density of $W$ are obtained for moderate values of $S$ and $N$, so this does not pose a serious limitation in practice. Figure~\ref{fig:pascal} illustrates the growth of condition numbers for square, rectangular, and rescaled Pascal matrices.

\begin{figure}[htb]
    \centering
    \includegraphics[scale=.6]{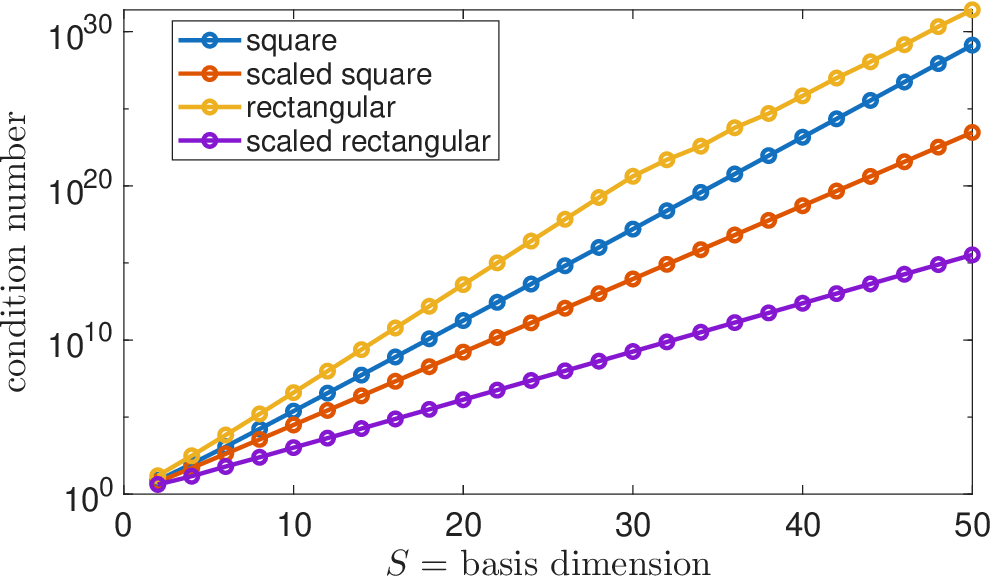}
    \caption{
     Condition numbers as a function of the basis dimension $S$.
The blue curve corresponds to $(S+1)\times(S+1)$ Pascal matrices, and the red curve
to their row-scaled versions. The yellow curve corresponds to $2(S+1)\times(S+1)$
Pascal matrices, and the purple curve to their scaled counterparts.
Row scaling significantly improves conditioning for both square and rectangular
matrices. For a fixed basis dimension $S$, the scaled rectangular matrices exhibit
the smallest condition numbers and are therefore used in our algorithm}
    \label{fig:pascal}
\end{figure}


\paragraph{Estimating $\beta$.}The parameter $\beta$ characterizes the decay of the right tail of the density of $W$. While lower bounds for this decay rate have been proposed in the literature (see, for instance,~\cite{Fernley2024}), these bounds might be too loose in practice. For this reason, we estimate $\beta$ empirically, following the strategy described in Section~\ref{sec:setup} below.

\section{Numerical examples}\label{sec:numericalexamples}

We implemented the proposed algorithm in \textsc{Matlab}; the code is available at \url{https://github.com/numpi/supercritical-galton-watson}, together with scripts to reproduce all numerical experiments reported in this paper. 

\subsection{Setup of the numerical experiments}\label{sec:setup}
For the different offspring distribution $P(z)$ considered below, we approximate the first $N+1=80$ coefficients of the Taylor expansion of the Laplace-Stieltjes transform $\varphi(z)$ of $W$ using Newton's method (Algorithm~\ref{alg:newton}). We use the initial vector $\varphi^{(0)} = \begin{bmatrix}1 & -1 & 0 &  \ldots & 0\end{bmatrix}^\top$ and a stopping tolerance of $10^{-14}$. 

The extinction probability $q$ is computed as the smallest real solution of the
equation $P(z)=z$, obtained using \textsc{Matlab}’s \texttt{roots} function for
polynomials. We then compute the parameter $\alpha$ using~\eqref{eq:a}.
To estimate the parameter $\beta$, we adopt the following heuristic procedure:
\begin{enumerate}
    \item We simulate the Galton-Watson process $M$ times until generation $T$; unless stated otherwise, we take $M = 100,000$ and $T = 12$.
    \item  Using the simulation output, we construct a histogram-based approximation of the density of $W$ via the \textsc{Matlab} built-in function \texttt{histogram}, which uses bins of uniform width.
    \item We estimate $\beta$ by fitting a linear regression to the logarithm of the last $30\%$ of the histogram bins, using \textsc{Matlab}’s \texttt{polyfit}, and take the estimated slope as our value of $\beta$.   
\end{enumerate}
Finally, we compute the first $40$ coefficients of~\eqref{eq:fN} by solving the rescaled version of the least-squares problem~\eqref{eq:linsysequiv}.

\subsection{Four examples}\label{sec:test1}
We test our algorithms on four GW processes with the following polynomial offspring probability generating functions:
\begin{itemize}
    \item $P_1(z) = 0.3z + 0.4z^2 + 0.2z^3 + 0.1z^4$;
    \item $P_2(z) = 0.5z + 0.3z^3 + 0.2z^4$;
    \item $P_3(z) = 0.1 + 0.1z + 0.1z^2 + 0.1z^3 + \ldots + 0.1 z^9$;
    \item $P_4(z) = 0.1 + 0.5z + 0.2z^3 + 0.1 z^4 + 0.1z^5$.
\end{itemize}
Using~\eqref{eq:a}, we find that the parameter $\alpha$ is positive for $P_1$ and $P_3$, and negative for $P_2$ and $P_4$. Therefore,  
these four examples illustrate different qualitative behaviors of the density of $W$: cases where $\lim_{x\to0^+}f(x)=0$ ($P_1$ and $P_3$), cases where $\lim_{x\to0^+}f(x)=\infty$ ($P_2$ and $P_4$), and situations with ($P_3$ and $P_4$) or without ($P_1$ and $P_2$) a point mass at $0$. The coefficients of $\varphi$ are computed via Newton's method and the density is approximated with the techniques described in Section~\ref{sec:momentmatching}.

Figures~\ref{fig:examples12} and~\ref{fig:examples34} display, in red, the proposed approximations of the density of $W$. 
The histograms represent the outputs of the simulations. The close agreement between the empirical and approximated densities indicates that the proposed method provides an accurate reconstruction of the distribution of $W$. Note that, in the cases in which $p_0 \neq 0$, there is a nonzero probability of extinction; this means that, in the histograms in Figure~\ref{fig:examples34}, the vertical bar in $0$ is approximating a delta function, and it is \emph{not} approximating the value of $q$.

\begin{figure}[htb]
    \begin{subfigure}[b]{0.48\textwidth}
		\centering
    \includegraphics[width=\textwidth]{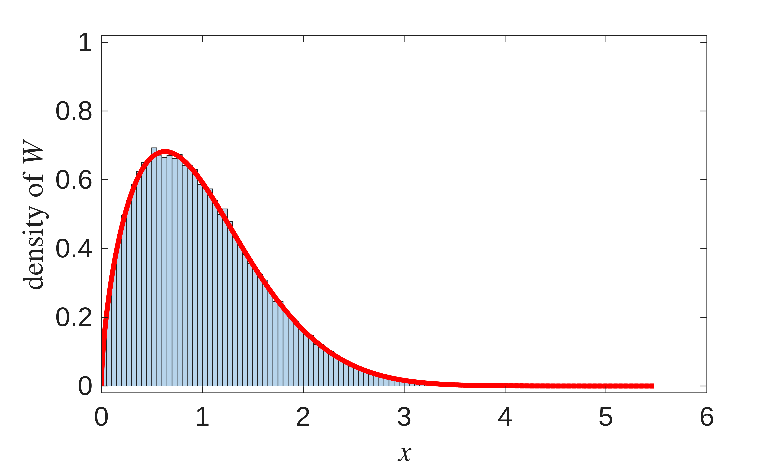}
    \end{subfigure}
    \begin{subfigure}[b]{0.48\textwidth}
		\centering
    \includegraphics[width=\textwidth]{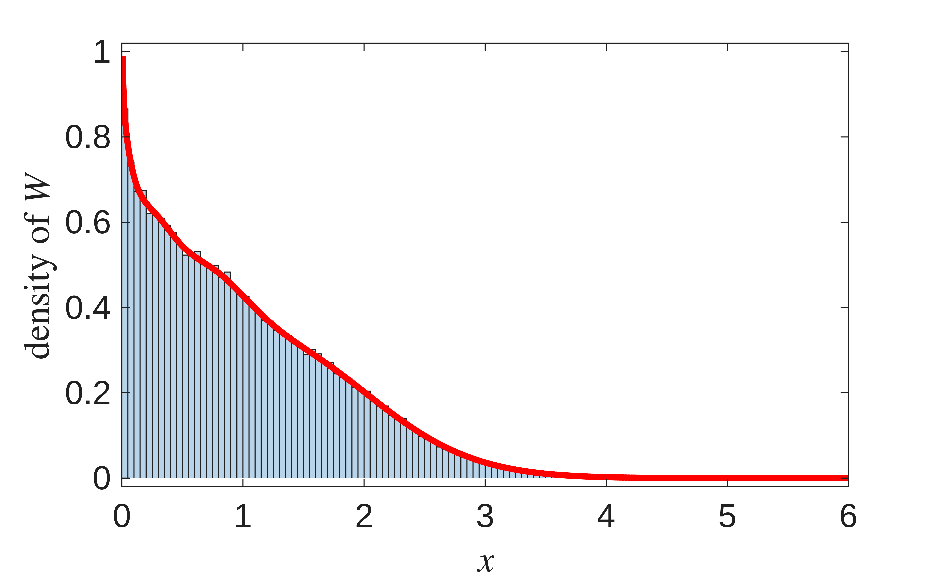}
    \end{subfigure}
    \caption{Approximations of the density of $W$ for GW processes with offspring p.g.f. $P_1(z)$ (left) and $P_2(z)$ (right). In both cases, the extinction
    probability is $q=0$, and hence there is no point mass at $0$.}
    \label{fig:examples12}
\end{figure}

\begin{figure}[htb]
    \begin{subfigure}[b]{0.48\textwidth}
		\centering
    \includegraphics[width=\textwidth]{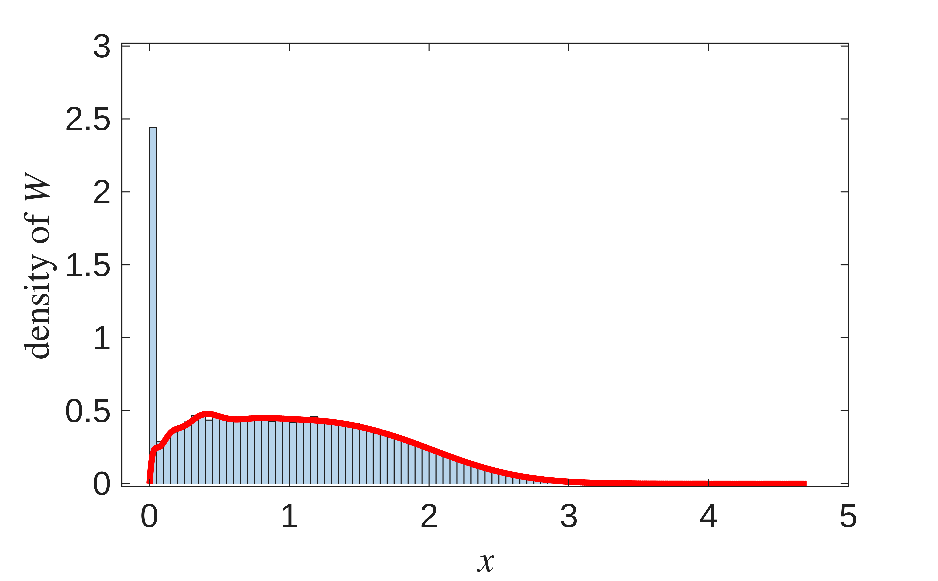}
    \end{subfigure}
    \begin{subfigure}[b]{0.48\textwidth}
		\centering
    \includegraphics[width=\textwidth]{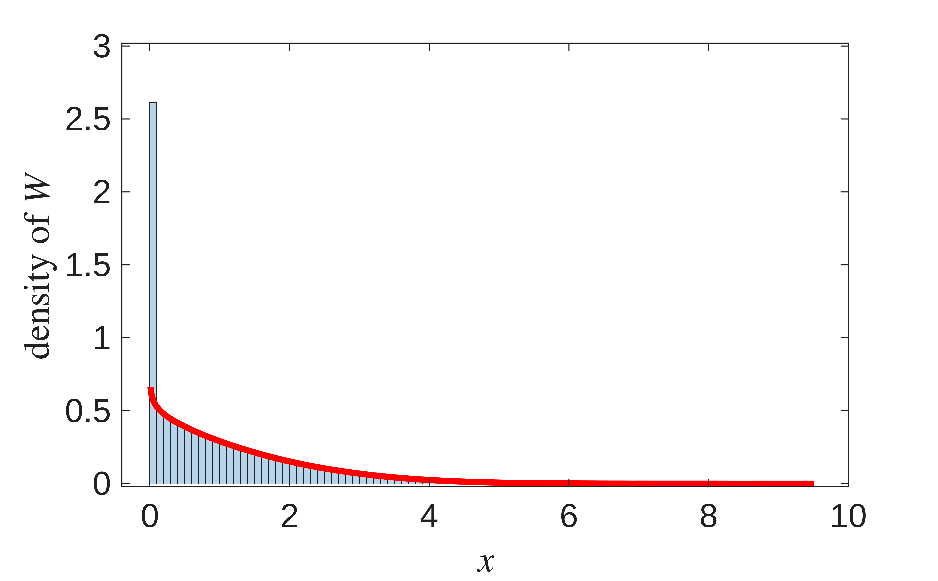}
    \end{subfigure}
    \caption{Approximations of the density of $W$ for GW processes with offspring p.g.f. $P_3(z)$ (left) and $P_4(z)$ (right). For $P_3(z)$, a smaller number of generations ($T=8$) was used in the simulation. In both cases, the
    extinction probability is strictly positive; the red curve represents our
    approximation of the absolutely continuous part of the density.}
    \label{fig:examples34}
\end{figure}


\subsection{Comparison with generalized Gamma distribution fitting}\label{sec:gamma}

Another approach for approximating the density of $W$ was proposed in \cite{Morris2024} and consists in fitting a generalized Gamma distribution by matching moments of $W$. More precisely, the density is approximated by
\begin{equation}\label{eq:generalized_gamma}
f(x) \approx
\frac{c_3}{c_1^{c_2}\Gamma(c_2/c_3)}\,
x^{c_2-1}
\exp\!\left[-\left(\frac{x}{c_1}\right)^{c_3}\right],
\end{equation}
where the parameters $c_1, c_2, c_3$ are chosen by minimizing 
a suitably weighted moment-matching loss function; see~\cite[Section~3.6]{Morris2024}.
Note that this approach may be applied, in principle, whenever the moments of $W$ are known; hence, it is not limited to quadratic offspring generating functions, unlike other parts of~\cite{Morris2024}. 

We compare this strategy with our approximation using Laguerre polynomials~\eqref{eq:fN} on two examples corresponding to the offspring distributions $P_1(z)$ and $P_3(z)$ from Section~\ref{sec:test1}. For both cases, we compute the moments of $W$ via Newton's method as in Section~\ref{sec:test1}, and we use the code provided in~\url{https://github.com/djmorris7/Computation_of_random_time-shifts} to estimate the parameters in~\eqref{eq:generalized_gamma}. We test the use of $5$, $10$, and $20$ moments. 
In practice, we observe that using a larger number of moments makes
the optimization problem increasingly difficult, and often degrades the quality of
the resulting approximation. In contrast, our method~\eqref{eq:fN} consistently uses $80$ moments and $40$ basis functions.


The results are shown in Figure~\ref{fig:gamma}. For the example based on $P_1(z)$
(left), using only five moments already produces a reasonable approximation with the generalized Gamma model, while increasing the number of moments leads to poorer fits. For the example based on $P_3(z)$ (right), the generalized Gamma approximation struggles to reproduce the shape of the density of $W$, notably due to the presence of two local maxima. In both cases, the Laguerre-based approximation~\eqref{eq:fN} follows the empirical histogram more closely.

Overall, we expect the approximation~\eqref{eq:fN} to be more accurate in general,
as it relies on a richer and more flexible class of approximating functions than the three-parameter generalized Gamma family.

\begin{figure}[htb]
    \begin{subfigure}[b]{0.48\textwidth}
		\centering
    \includegraphics[width=\textwidth]{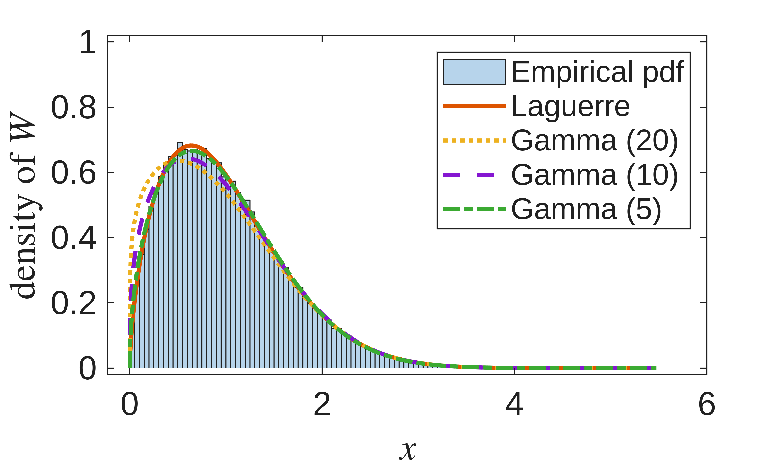}
    \end{subfigure}
    \begin{subfigure}[b]{0.48\textwidth}
		\centering
    \includegraphics[width=\textwidth]{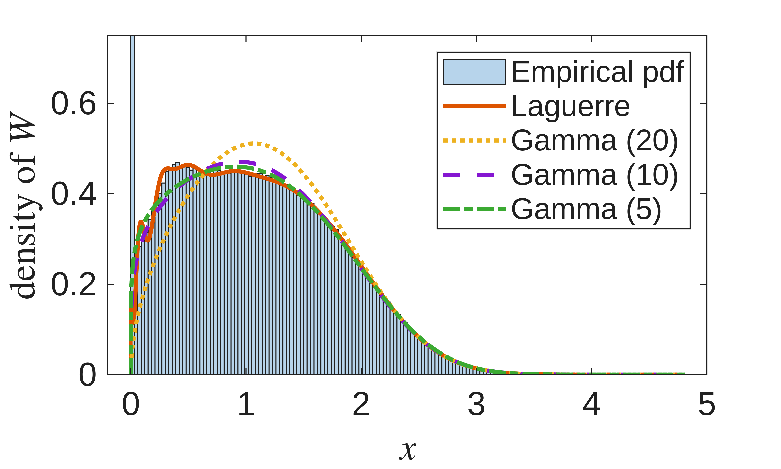}
    \end{subfigure}
    \caption{Comparison of the Laguerre-based approximation~\eqref{eq:fN} with the
    generalized Gamma approximation~\eqref{eq:generalized_gamma} for offspring
    distributions $P_1(z)$ (left) and $P_3(z)$ (right). Histograms represent empirical
    densities obtained from simulations.}
    \label{fig:gamma}
\end{figure}

\subsection{Estimating early fluctuations in bird populations}\label{sec:birds}

We apply our methodology to estimate the density of the random variable $W$ associated with the early growth dynamics of two bird species: the whooping crane and the Chatham Island black robin. 

When fitting a GW process to yearly population census data of the whooping crane in~\cite[Section 2.1]{braunsteins2025existence}, the resulting offspring distribution is 
\begin{equation*}
    P(z) = 0.1538 + 0.6491z +0.1971 z^2.
\end{equation*}
This corresponds to a mean offspring number $m \approx 1.04$ and an extinction probability $q \approx 78.03\%$. 

On the other hand, the offspring distribution of the Chatham Island black robin considered in~\cite[Section 6]{braunsteins2025consistent}, obtained by setting the
probability of successful nesting to $r=1$ in order to remove population-size dependence, is given by
\begin{equation*}
P(z) = 0.1036 + 0.3551z + 0.3448 z^2 + 0.1553 z^3 + 0.0366 z^4 + 0.0044 z^5 + 0.0002z^6. 
\end{equation*}
This distribution has a mean offspring number $m \approx 1.68$ and an extinction probability $q \approx 17.93\%$.

The approximated densities of the random variable $W$ corresponding to the GW models for these two birds species are shown in \rev{Figures~\ref{fig:whoopingcrane} and~\ref{fig:blackrobin}}. As discussed in Section~\ref{sec:supercritical}, approximating the distribution of $W$ provides information on establishment times and on the distribution of the population size at a fixed generation. 

For instance, conditional on non-extinction, the $90\%$ prediction interval for  the population size in the $30$-th generation, $Z_{30}$, is approximately $[0.96,   47.23]$ for the whooping crane population. Similarly, the 90\% prediction interval for the size of the $10$th generation, $Z_{10}$, conditional on non-extinction, is approximately $[26.89,\,532.21]$ for the black robin population. 

Finally, Figure~\ref{fig:birds} displays the distribution of the approximate establishment time $\tau_{100}$ for the two species, computed using~\eqref{eq:density_establishment_time}. We observe in \rev{Figures~\ref{fig:whoopingcrane} and~\ref{fig:blackrobin}} that the distribution of 
$W$ exhibits a wider spread for the whooping crane population than for the black robin population. Since $\mathbb E[W]=1$
 in both cases, this increased variability is consistent with a larger probability mass near zero for the whooping crane population, which in turn leads to longer typical establishment times, as illustrated in Figure~\ref{fig:birds}. 

\begin{figure}[htb]
    \begin{subfigure}[b]{0.42\textwidth}
		\centering
    \includegraphics[width=\textwidth]{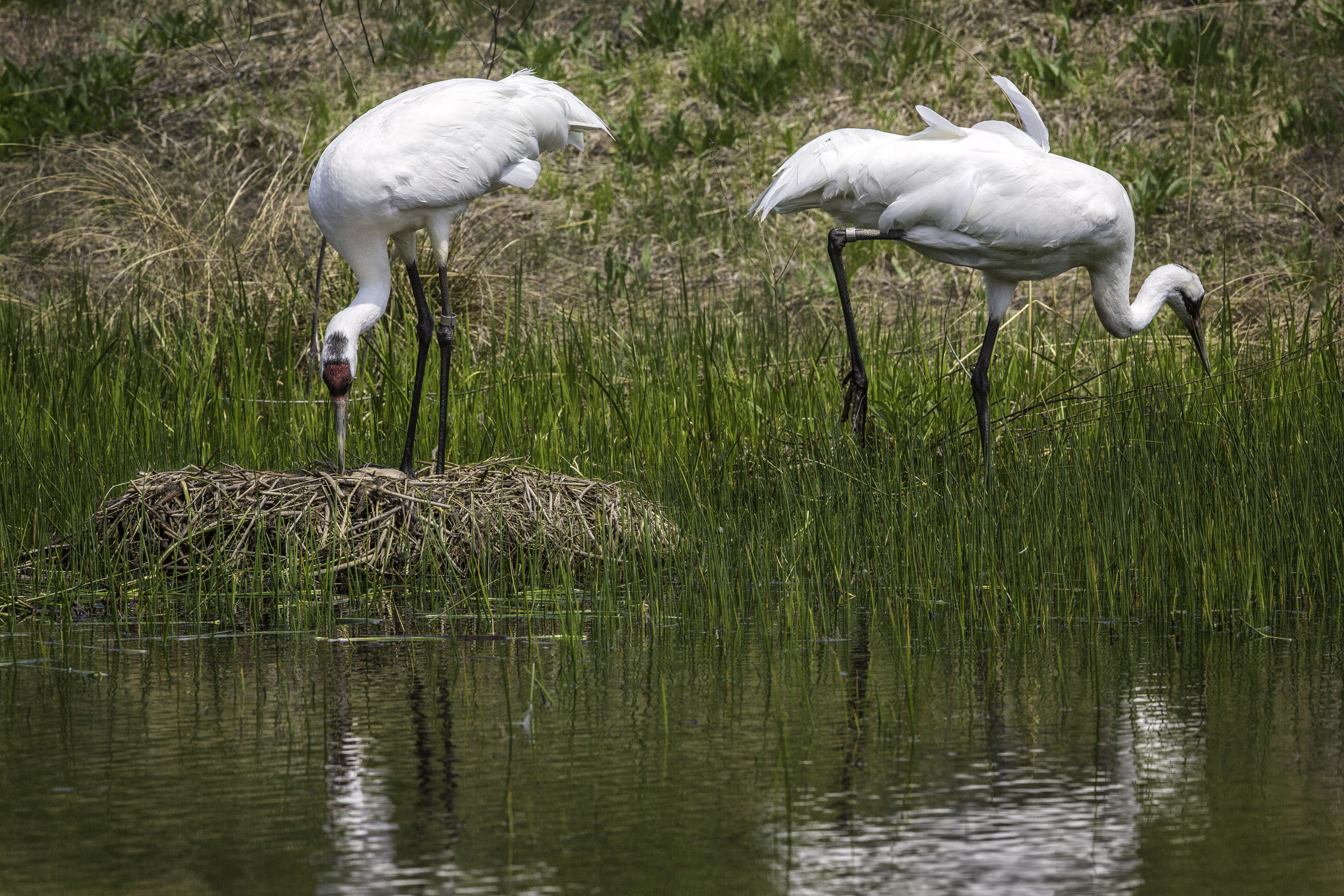}
    \end{subfigure}
    \begin{subfigure}[b]{0.54\textwidth}
		\centering
    ~
    \vspace{-4mm}
    ~\includegraphics[width=\textwidth]{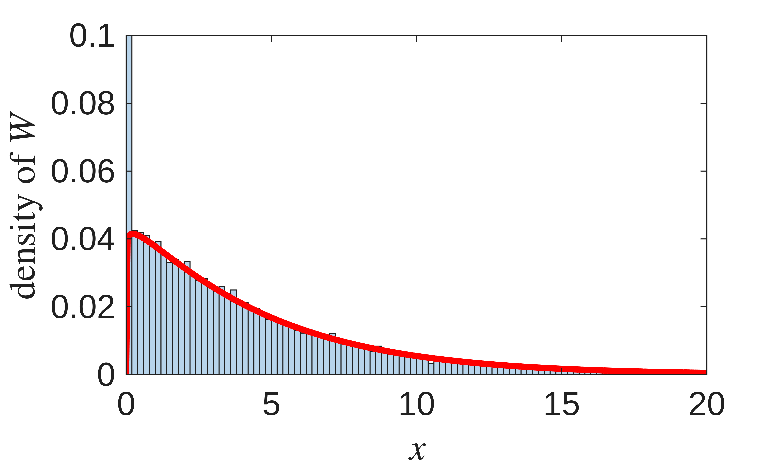}
    \end{subfigure}
    \caption{Plot corresponding to the density of $W$ for the whooping crane. For this example, we have $m \approx 1.04$ and the distribution has a heavier right tail compared to the previous examples. For this reason, to estimate the parameter $\beta$ we have used the $40\%$-to-$70\%$ range of the histogram, resulting in $\beta \approx 0.246$. The simulation was run until time $T = 100$.}
    \label{fig:whoopingcrane}
\end{figure}

\begin{figure}[htb]
    \begin{subfigure}[b]{0.42\textwidth}
		\centering
    \includegraphics[width=0.95\textwidth]{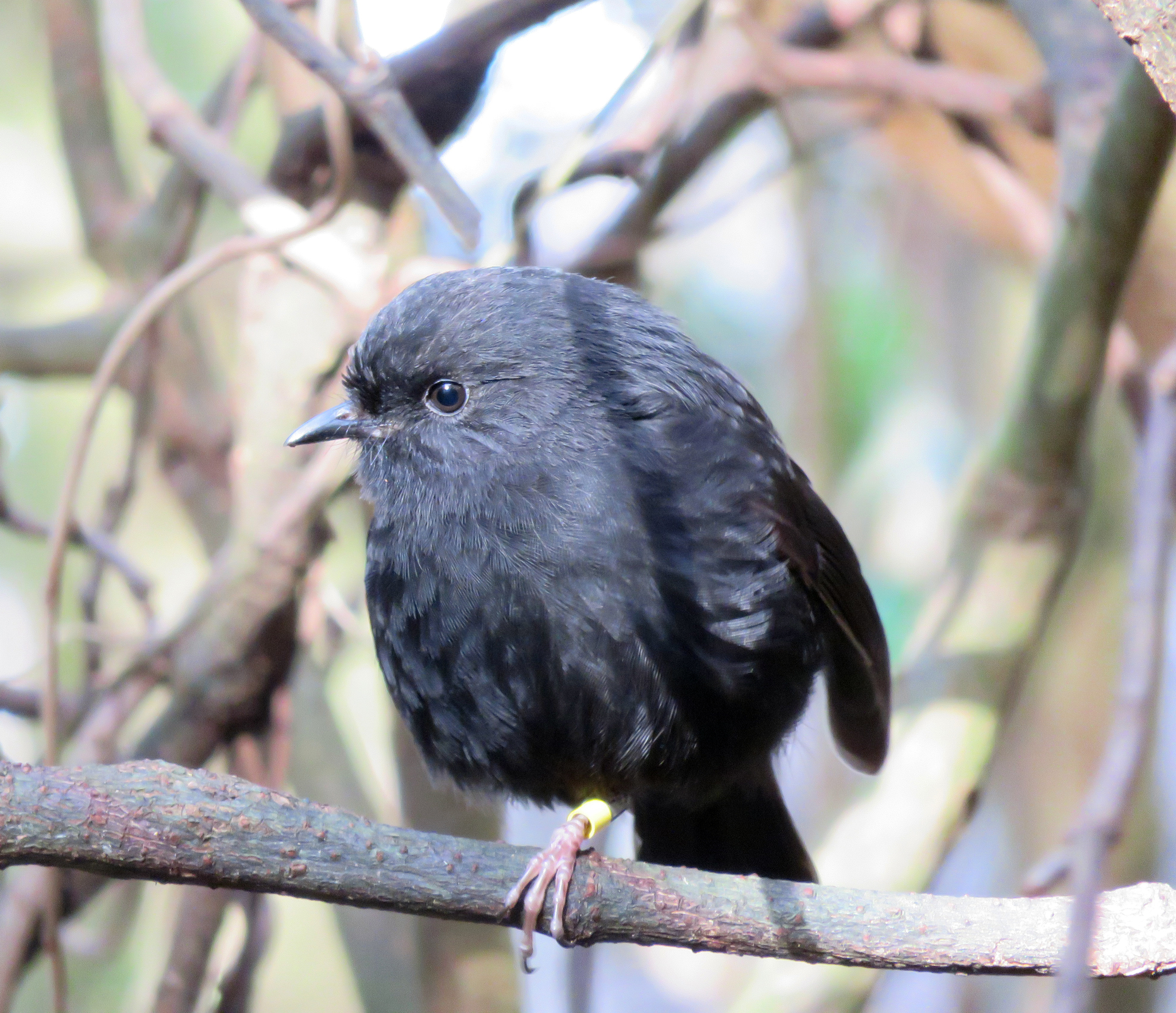}
    \end{subfigure}
    \begin{subfigure}[b]{0.54\textwidth}
		\centering
    \includegraphics[width=\textwidth]{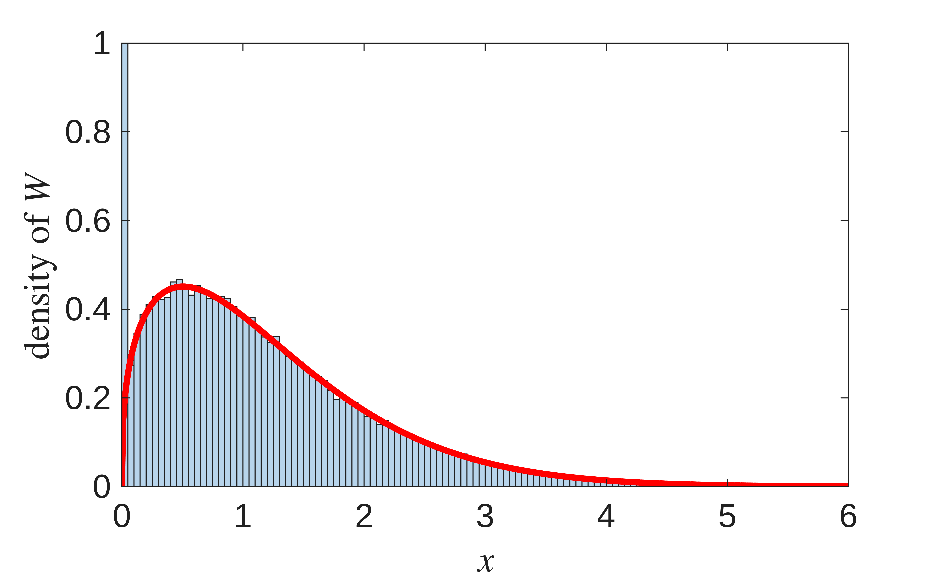}
    \end{subfigure}
    \caption{Plot corresponding to the density of $W$ for the Chatham Islands black robin population. For this example, we have $m \approx 1.68$ and the estimated parameter $\beta$ is $0.543$. Simulations were run until $T = 15$.}
    \label{fig:blackrobin}
\end{figure}


\begin{figure}[htb]
    \centering
    \includegraphics[width=.5\textwidth]{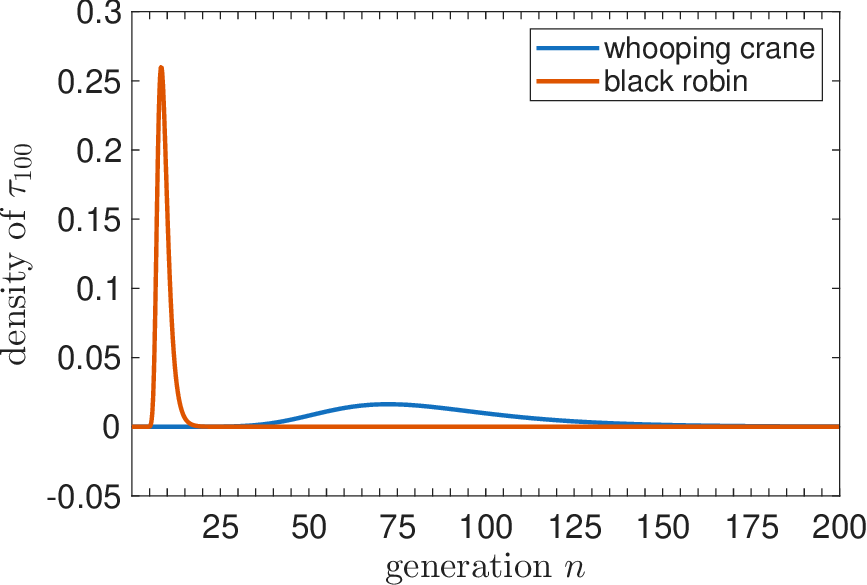}
    \caption{Approximate density of the time at which the population size first reaches $100$ individuals for the whooping crane and black robin populations; see Section~\ref{sec:birds}.}
    \label{fig:birds}
\end{figure}

\subsection{Handling non-polynomial offspring probability generating functions}\label{sec:lim}

When the Laplace-Stieltjes transform $\varphi(z)$ of the random variable $W$ is not entire, there is no guarantee that the techniques of Section~\ref{sec:poincare} based on the fixed-point iteration~\eqref{eq:infinite-fixed-point} will apply. However, one could approximate the offspring probability generating function  by a polynomial -- for instance, by truncating its power series expansion, if it exists -- and then apply one of the methods in Section~\ref{sec:poincare} in the hope of obtaining useful information. 

We explore this idea in the setting of linear fractional generating functions, which correspond to modified geometric offspring distributions and constitute one of the very few cases in which the distribution of $W$ is known explicitly; see~\cite[Chapter 1]{Harris1963}. Given real numbers $c \in (0,1)$ and $b \in (0, 1-c]$, the probability generating function of the offspring distribution is
\begin{equation}\label{eq:Plinearfractional}
P(z) = 1 - \frac{b}{1-c} + \frac{bz}{1-cz} = 1 - \frac{b}{1-c} + \sum_{i=1}^{\infty} b c^{i-1} z^i.
\end{equation}
We concentrate on the case $b = 1-c$, which corresponds to zero extinction probability. In this case, the Laplace-Stieltjes transform of $W$ and its density are known exactly and do not depend on $c$:
\begin{equation}\label{eq:exp}
\varphi(z) = \frac{1}{1+z}, \qquad f(x) = e^{-x}.
\end{equation}
The Taylor coefficients of $\varphi$ are given by $\varphi_i = (-1)^i$ for all $i \ge 0$, and $\varphi$ is analytic only in the unit disk of the complex plane.

Although $P(z)$ is not a polynomial, we can apply the forward substitution method of Section~\ref{sec:forward} without modifications, since evaluating~\eqref{eq:forward} only requires computing the functions $P$ and $P'$ applied to Toeplitz matrices. In contrast, applying Newton's method (or the fixed-point iteration) requires truncating the series~\eqref{eq:Plinearfractional} to a finite number of terms. Here we truncate the series to its first $80$ terms and we run our algorithm for two values of $c$ in $\{0.7, 0.9\}$.  

Figure~\ref{fig:LF} \rev{(left)} shows the coefficients of $\varphi$ approximated by Newton's method and by forward substitution. We observe that for the smaller value of $c$, more coefficients are closer to their correct value. 
In both cases, Newton’s method exhibits stagnation, and increasing the number of iterations does not improve the approximation. A possible explanation for this behavior is that the coefficients of the offspring generating function decay more quickly when $c$ is small, making the truncated generating function ``closer'' to a polynomial of low degree. 

By contrast, the forward substitution method is able to recover the correct coefficients of $\varphi(z)$. More generally, it is applicable to all offspring probability generating functions that are analytic around $0$, without requiring truncation.

\rev{In Figure~\ref{fig:LF} (right), we consider the case $c = 0.7$ and show the approximation error, measured in the Kolmogorov distance, between the exact solution~\eqref{eq:exp} and our approximations, obtained from the coefficients of $\varphi$ computed via Newton's method (plain line) and the forward substitution method (dashed line). For the estimation of the parameter $\beta$ (required by our method), we follow the strategy described in Section~\ref{sec:setup}, using the last $40\%$ of the histogram bins obtained by simulating $M = 10^4$ samples with $T = 10$, which yields $\beta \approx 0.2$. }

\rev{On the same plot we also report the approximation error between the exact solution and the empirical cumulative distribution function (CDF) of the scaled population size at generation $g$, $W_g$, obtained by Monte Carlo simulations, for $g = 1, \ldots, 10$ (star-marked line). For each value of $g$, we generated $10^6$ samples of $W_g$. The Kolmogorov distance is computed as the maximum discrepancy between discretised CDFs over the interval $[0,10]$, using $10^7$ equispaced points.}


\rev{The convergence of the distribution of $W_g$ to the true solution levels off beyond generation $g=6$. This behavior is due to the statistical error in the Monte Carlo approximation of the CDF of $W_g$, which is of order $\mathcal{O}(1/\sqrt{M}) \approx 10^{-3}$, where $M=10^6$ denotes the number of samples.
Increasing the number of samples, especially for larger values of $g$, becomes computationally expensive. By contrast, our method, with both Newton and forward substitution, achieves a Kolmogorov distance of approximately $10^{-4}$ while requiring a substantially smaller number of simulations ($10^4$ samples), used only for the estimation of $\beta$.
}



\begin{figure}[htb]
\begin{subfigure}[b]{0.42\textwidth}
		\centering    \includegraphics[scale=.5]{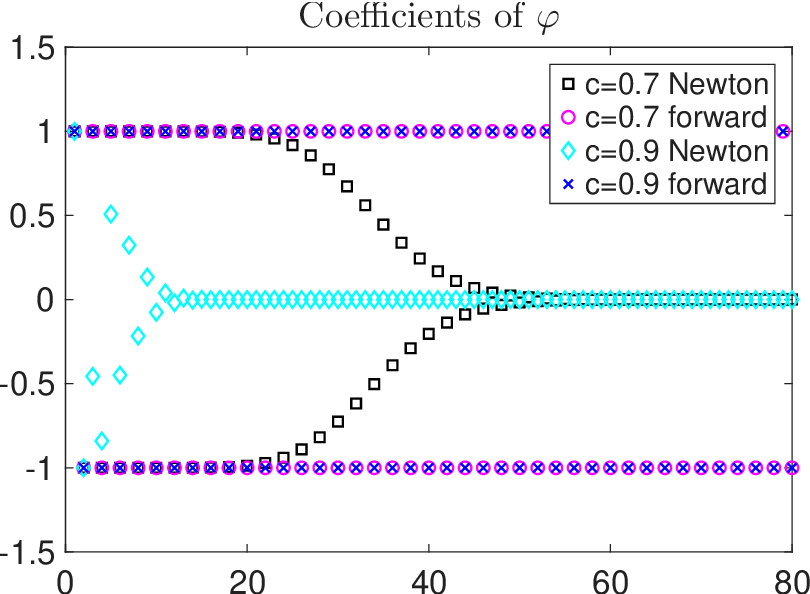}
\end{subfigure}
\begin{subfigure}[b]{0.42\textwidth}
		\centering
~
        \vspace{-5mm}
        \includegraphics[scale=.53]{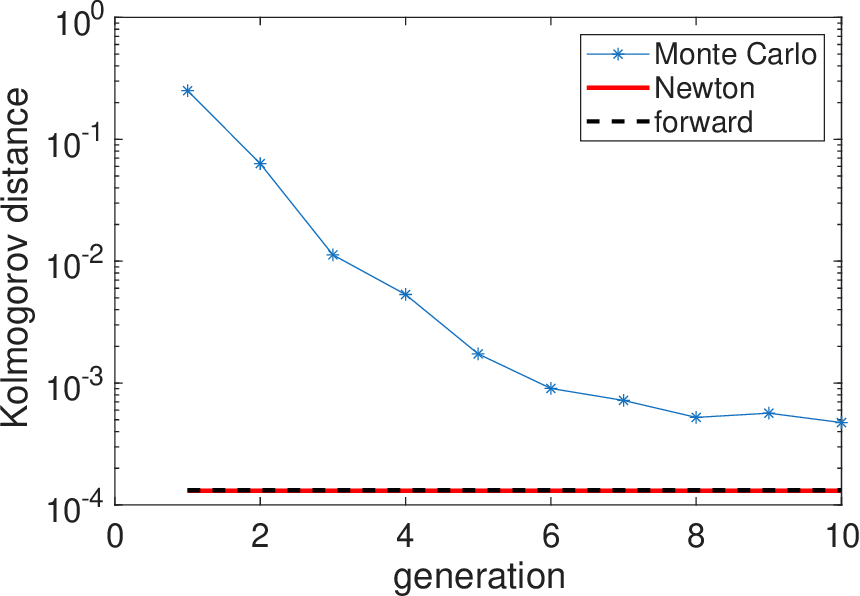}
        \end{subfigure}

        \vspace{3mm}
    \caption{\rev{Left:} approximations of the coefficients of $\varphi$ obtained by Newton's method and the forward substitution method for the two linear fractional examples considered in Section~\ref{sec:lim}. Note that the forward substitution method achieves very precise approximations in both cases. \rev{Right: Kolmogorov distance between the exact distribution of the linear fractional model with $c = 0.7$ and the approximations obtained with three methods. The plain red line and the dashed black line refer to our method using coefficients of $\varphi$ generated by Newton's method and the forward substitution method, respectively. The star-marked blue line represents the error associated to the empirical distribution obtained after $1,2,\ldots,10$ generations using $10^6$ Monte Carlo samples.}} 
    \label{fig:LF}
\end{figure}


\section{Conclusions and outlook}
We have proposed a practical numerical method for estimating the density of the random variable $W$ arising as the scaling limit of supercritical single-type Galton-Watson branching processes. The method applies to a broad class of models in which the offspring
distribution has finite support. Through several numerical experiments, we have shown that the proposed approach is computationally efficient -- requiring simulations only for a secondary task, namely the estimation of the tail parameter $\beta$ -- and yields accurate approximations of the density of $W$ across a variety of regimes.

The extension of the approach to multi-type processes appears conceptually
straightforward: we would need to solve the analogue of the functional equation~\eqref{eq:poincare} in the space of multivariate power series, and approximate  the density of $W$ using linear combinations of tensor products of Laguerre polynomials with exponential damping. In practice, however, mitigating the curse of dimensionality
and ensuring scalability to models with a large number of types poses significant challenges. Developing efficient strategies to address these issues constitutes a natural direction for future research.

\section*{Acknowledgements}
The authors thank the anonymous referees for their careful reading and helpful comments.
Sophie Hautphenne would like to thank the Australian Research Council (ARC) for support through the Discovery Project DP200101281. \rev{The work of Stefano Massei has been supported by the PRIN 2022
Project ``Low-rank Structures and Numerical Methods in Matrix and Tensor Computations and their
Application''. The work of Alice Cortinovis has been partially supported by the PRIN 2022 Project ``MOLE: Manifold constrained Optimization and LEarning''.}

\bibliographystyle{abbrv}
\bibliography{biblio}

@article {Dubuc1971,
    AUTHOR = {Dubuc, M. Serge},
     TITLE = {La densit\'e{} de la loi-limite d'un processus en cascade
              expansif},
   JOURNAL = {Z. Wahrscheinlichkeitstheorie und Verw. Gebiete},
  FJOURNAL = {Zeitschrift f\"ur Wahrscheinlichkeitstheorie und Verwandte
              Gebiete},
    VOLUME = {19},
      YEAR = {1971},
     PAGES = {281--290},
   MRCLASS = {60J80},
  MRNUMBER = {300353},
MRREVIEWER = {E.\ Seneta},
       DOI = {10.1007/BF00535833},
       URL = {https://doi.org/10.1007/BF00535833},
}

@article {Fernley2024,
    AUTHOR = {Fernley, John and Jacob, Emmanuel},
     TITLE = {A universal right tail upper bound for supercritical
              {G}alton-{W}atson processes with bounded offspring},
   JOURNAL = {Statist. Probab. Lett.},
  FJOURNAL = {Statistics \& Probability Letters},
    VOLUME = {209},
      YEAR = {2024},
     PAGES = {Paper No. 110082, 5},
      ISSN = {0167-7152,1879-2103},
   MRCLASS = {99-01},
  MRNUMBER = {4708857},
       DOI = {10.1016/j.spl.2024.110082},
       URL = {https://doi.org/10.1016/j.spl.2024.110082},
}

@article {Provost2012,
    AUTHOR = {Provost, Serge B. and Jiang, Min},
     TITLE = {Orthogonal polynomial density estimates: alternative
              representation and degree selection},
   JOURNAL = {Int. J. Comput. Math. Sci.},
  FJOURNAL = {International Journal of Computational and Mathematical
              Sciences},
    VOLUME = {6},
      YEAR = {2012},
     PAGES = {17--24},
      ISSN = {2070-3910},
   MRCLASS = {62G07 (33C45 42C05)},
  MRNUMBER = {2844957},
}

@book {Harris1963,
    AUTHOR = {Harris, Theodore E.},
     TITLE = {The theory of branching processes},
    SERIES = {Die Grundlehren der mathematischen Wissenschaften},
    VOLUME = {Band 119},
 PUBLISHER = {Springer-Verlag, Berlin; Prentice Hall, Inc., Englewood
              Cliffs, NJ},
      YEAR = {1963},
     PAGES = {xiv+230},
   MRCLASS = {60.67},
  MRNUMBER = {163361},
MRREVIEWER = {P.\ A. P. Moran},
}

@article{Morris2024,
  title={Computation of random time-shift distributions for stochastic population models},
  author={Morris, Dylan and Maclean, John and Black, Andrew J},
  journal={Journal of Mathematical Biology},
  volume={89},
  number={3},
  pages={33},
  year={2024},
  publisher={Springer}
}

@book{Athreya2012,
  title={Branching processes},
  author={Athreya, Krishna B and Ney, Peter E},
  volume={196},
  year={1972},
  publisher={Springer Science \& Business Media}
}

@article{Edelman2004,
  TITLE={Pascal matrices},
  AUTHOR={Edelman, Alan and Strang, Gilbert},
  JOURNAL={The American Mathematical Monthly},
  VOLUME={111},
  number={3},
  pages={189--197},
  YEAR={2004},
  publisher={Taylor \& Francis}
}

@book{Higham2002,
  title={Accuracy and stability of numerical algorithms},
  author={Higham, Nicholas J},
  year={2002},
  publisher={SIAM}
}

@book{Valiron1954,
  author    = {Georges Valiron},
  title     = {Fonctions analytiques},
  publisher = {Presses Universitaires de France},
  address   = {Paris},
  year      = {1954}
}

@article{Poincare1890,
  title={Sur une classe nouvelle de transcendantes uniformes},
  author={Poincar{\'e}, Henri},
  journal={Journal de Math{\'e}matiques pures et appliqu{\'e}es},
  volume={6},
  pages={313--365},
  year={1890}
}

@inproceedings{Derfel2007,
  author    = {Gregory Derfel and Peter J. Grabner and Fritz Vogl},
  title     = {Asymptotics of the {P}oincaré functions},
  booktitle = {Probability and Mathematical Physics: A Volume in Honor of Stanislav Molchanov},
  editor    = {D. Dawson and V. Jakšić and B. Vainberg},
  series    = {CRM Proceedings and Lecture Notes},
  volume    = {42},
  pages     = {113--130},
  publisher = {Centre de Recherches Mathématiques, Montréal},
  year      = {2007}
}

@inproceedings{kinoshita2024power,
  title={Power series composition in near-linear time},
  author={Kinoshita, Yasunori and Li, Baitian},
  booktitle={2024 IEEE 65th Annual Symposium on Foundations of Computer Science (FOCS)},
  pages={2180--2185},
  year={2024},
  organization={IEEE}
}

@article{braunsteins2025existence,
  title={Existence and non-existence of consistent estimators in supercritical controlled branching processes},
  author={Braunsteins, Peter and Hautphenne, Sophie and Kerlidis, James},
  journal={arXiv preprint arXiv:2504.03389},
  year={2025}
}

@article{braunsteins2025consistent,
  title={Consistent least squares estimation in population-size-dependent branching processes},
  author={Braunsteins, Peter and Hautphenne, Sophie and Minuesa, Carmen},
  journal={Journal of the American Statistical Association},
  number={just-accepted},
  pages={1--21},
  year={2025},
  publisher={Taylor \& Francis}
}

@book{duren1970theory,
  title={Theory of $H^p$ Spaces},
  author={Duren, Peter L},
  volume={38},
  year={1970},
  publisher={Academic press}
}

@article{barbour2016emergence,
  title={On the emergence of random initial conditions in fluid limits},
  author={Barbour, Andrew D and Chigansky, Pavel and Klebaner, Fima C},
  journal={Journal of Applied Probability},
  volume={53},
  number={4},
  pages={1193--1205},
  year={2016},
  publisher={Cambridge University Press}
}

@article{baker2018persistence,
  title={Persistence of small noise and random initial conditions},
  author={Baker, Jeremy and Chigansky, Pavel and Hamza, Kais and Klebaner, Fima C},
  journal={Advances in Applied Probability},
  volume={50},
  number={A},
  pages={67--81},
  year={2018},
  publisher={Cambridge University Press}
}

@article{bauman2023approximation,
  title={An approximation of populations on a habitat with large carrying capacity. arXiv},
  author={Bauman, N and Chigansky, P and Klebaner, F},
  journal={arXiv preprint arxiv:2303.03735},
  year={2023}
}

@article{chigansky2018can,
  title={What can be observed in real time {PCR} and when does it show?},
  author={Chigansky, Pavel and Jagers, Peter and Klebaner, Fima C},
  journal={Journal of Mathematical Biology},
  volume={76},
  number={3},
  pages={679--695},
  year={2018},
  publisher={Springer}
}

@article{hautphenne2020low,
  title={A Low-Rank Technique for Computing the Quasi-Stationary Distribution of Subcritical {G}alton--{W}atson Processes},
  author={Hautphenne, Sophie and Massei, Stefano},
  journal={SIAM Journal on Matrix Analysis and Applications},
  volume={41},
  number={1},
  pages={29--57},
  year={2020},
  publisher={SIAM}
}
\end{document}